\documentclass{amsart}
\usepackage{marginnote}

\usepackage{amscd,xcolor}
\usepackage{tikz}

\usepackage{enumerate,amssymb,amsmath,graphics,epsfig}

\usepackage{hyperref}
\usepackage{subfigure}
\usepackage{hhline}
\usepackage{caption} 
\usepackage{graphicx}

\newcommand\bfu{\mathbf{u}}
\newcommand\bfv{\mathbf{v}}
\newcommand\bfz{\mathbf{z}}
\newcommand\bfn{\mathbf{n}}
\newcommand\bft{\mathbf{t}}
\newcommand\bff{\mathbf{f}}
\newcommand\bfw{\mathbf{w}}
\newcommand\bfx{\mathbf{x}}
\newcommand\bfzero{\textbf{0}}
\newcommand\bfpsi{\boldsymbol{\psi}}
\DeclareMathOperator\Div{\mathrm{div}}
\DeclareMathOperator\Rot{\mathrm{rot}}
\DeclareMathOperator\Grad{\mathbf{grad}}
\DeclareMathOperator\Curl{\mathbf{curl}}
\newcommand\HrotE{\mathbf{H}(\Rot;E)}
\newcommand\HdivE{\mathbf{H}(\Div;E)}
\newcommand\Hrot{\mathbf{H}(\Rot;\Omega)}
\newcommand\Hroto{\mathbf{H}(\Rot^0;\Omega)}
\newcommand\Hdiv{\mathbf{H}(\Div;\Omega)}
\newcommand\Hodiv{\mathbf{H}_0(\Div;\Omega)}
\newcommand\Honezero{H^1_0(\Omega)}
\newcommand\Hone{H^1(\Omega)}

\newcommand\Hsbold{\mathbf{H}^s(\Omega)}

\newcommand\Ltwovecomega{\mathbf{L}^2(\Omega)}
\newcommand\Ltwovec{\mathbf{L}^2(E)}
\newcommand\RE{\mathbb{R}}
\newcommand\V{\mathbf{V}}
\newcommand\Vh{\mathbf{V}_h}
\newcommand\Vhn{\mathbf{V}_{h_n}}
\newcommand\Qh{Q_h}
\newcommand\Qhn{Q_{h_n}}
\newcommand\VE{\mathbf{V}_h^E}
\newcommand\T{\mathcal{T}}
\newcommand\Po{\mathcal{P}}
\newcommand\PE{\Pi_h^E}
\newcommand\bh{b_h}
\newcommand\bho{b_{h,0}}
\newcommand\bhs{b_{h,s}}
\newcommand\bE{\bho^E}
\newcommand\K{\boldsymbol{\mathcal{K}}}
\newcommand\Kerrot{{\K^{\Rot}}}
\newcommand\kerdiv{{\K_h^{\Div}}}
\newcommand\G{\boldsymbol{\mathcal{G}}}

\newcommand\Vhh{\widetilde{\mathbf{V}}_h}
\newcommand\Kbhrot{\K^{\Rot}_h}
\newcommand\Kbhnrot{\K^{\Rot}_{h_n}}
\newcommand\kerbh{\boldsymbol{\mathcal{K}}^b_h}
\newcommand\inter{\mathbb{K}_h}
\newcommand\uhn{\bfv_{{h}_n}}
\newcommand\uIn{\bfu^I_{{h}_n}}
\newcommand\vhn{\bfv_{{h}_n}}

\newcommand\bhn{b_{h_n}}
\newcommand\phn{p_{h_n}}
\newcommand\qhn{q_{h_n}}
\newcommand\ubar{\bar{\bfu}}
\newcommand{\m}[1]{\mathsf{#1}}
\newcommand\SE{S^E}
\newcommand\bfuSA{\widetilde\bfu_h^I}

\newtheorem{theorem}{Theorem}
\newtheorem{proposition}{Proposition}
\newtheorem{lemma}{Lemma}

\newtheorem{example}{Example}
\newtheorem{definition}{Definition}

\newtheorem{remark}{Remark}

\begin{document}
%======================================================================
\title[On the stabilization of VEM for an acoustic vibration problem]{On the
stabilization of a virtual element method for an acoustic
vibration problem}
%======================================================================
%===============================================================================

\author[L. Alzaben]{Linda Alzaben}
\address{King Abdullah University of Science and Technology (KAUST), Saudi Arabia}
\email{linda.alzaben@kaust.edu.sa}
\author[D. Boffi]{Daniele Boffi}
\address{King Abdullah University of Science and Technology (KAUST), Saudi Arabia
and University of Pavia, Italy}
\email{daniele.boffi@kaust.edu.sa}
\author[A. Dedner]{Andreas Dedner}
\address{University of Warwick, UK}
\email{A.S.Dedner@warwick.ac.uk}
\author[L. Gastaldi]{Lucia Gastaldi}
\address{University of Brescia, Italy}
\email{lucia.gastaldi@unibs.it}

%===============================================================================
%===============================================================================

\begin{abstract}

In this paper we introduce an abstract setting for the convergence analysis of
the virtual element approximation of an acoustic vibration problem. We discuss
the effect of the stabilization parameters and remark that in some cases it
is possible to achieve optimal convergence without the need of any
stabilization. This statement is rigorously proved for lowest order triangular
element and supported by several numerical experiments.

\end{abstract}

\keywords{Eigenvalue problem; Virtual element method; Stabilization free VEM;
Acoustic problem; Discrete compactness property}
\subjclass[2020]{65N25; 65N30; 70J30; 76M10}
\maketitle

%===============================================================================
%===============================================================================

\section{Introduction}

In this paper we study and analyze a virtual element method (VEM) introduced
in~\cite{Rodolfo} for the approximation of an acoustic vibration problem.

The use of VEM for the approximation of the solution to PDE eigenvalues
problems has been adopted and analyzed by several authors, starting from
standard elliptic problem~\cite{GardiniVacca}, including $hp$
VEM~\cite{Gardinihp}, nonconforming VEM~\cite{Gardininonc}, and mixed
schemes~\cite{VEMmixed}, the Steklov eigenvalue
problem~\cite{Steklov1,Steklov2,Steklov3,Steklov4}, plate
models~\cite{plate1,plate2,plate3,plate4}, linear
elasticity~\cite{elasticity}, to transmission
problems~\cite{transmission1,transmission2,transmission3,transmission4}
Recently,
it was observed that the presence of the stabilization parameters can be
problematic~\cite{Calcolo,Sema}. In particular, spurious modes can pollute the
spectrum and it could be difficult to rule them out unless the structure of
the exact solutions is known in advance.

After our prior investigations on the effect of the stabilization parameters
on the VEM numerical approximation of PDE eigenvalues problems, our attention
was drawn by the following sentences of~\cite[Section~5.3]{Rodolfo}:
``\emph{In the present case, no spurious eigenvalue was detected for any
choice of the stability constant. However, for large values of
[the stabilization parameters] $\sigma_E$, the
eigenvalues computed with coarse meshes could be very poor}'', and ``\emph{In
fact, it can be seen from this table that even the value $\sigma_E=0$ yields
very accurate results, in spite of the fact that for such a value of the
parameter the stability estimate and hence most of the proofs of the
theoretical results do not hold}''.

The aim of this paper is to study the convergence of the scheme proposed
in~\cite{Rodolfo} in an abstract theoretical setting and to discuss the effect
of the stabilization parameter.
By doing so, we make rigorous some of the statements appearing
in~\cite{Rodolfo}.
Our ultimate goal is to show that in some cases the stabilization is not
necessary (that is, $\sigma_E=0$ with the notation of the sentences above), to
provide numerical evidence of that, and to prove it rigorously when possible.

Parameter-free VEM schemes are the object of an intense and complex discussion
in the recent literature, starting from the pioneer work~\cite{Berrone2021}
and proceeding with~\cite{Berrone2022,Berrone2023,Lamperti,Chen}. The study of
parameter-free VEM is particular important in the case of eigenvalue problems,
where the presence of parameters can be source of spectral
pollution~\cite{Calcolo,Sema}. A first investigation on parameter free
eigenvalue problem is presented in~\cite{Meng}.

The structure of the paper is as follows: after some preliminary notation
given in Section~\ref{sec:function_spaces}, we describe the acoustic vibration
problem in Section~\ref{sec:the_continues_problem}. The discretization of the
problem is presented in Section~\ref{sec:the_virtual_element_discretization}
and an abstract theory for the approximation is developed in
Section~\ref{sec:abstract_theory}. The theory is based on an equivalent mixed
formulation of the problem, which allows to adapt the classical arguments
of~\cite{BBG2,Acta} to this situation. A crucial role is played by the
discrete compactness property~\cite{SDCP}.
As consequence of the abstract theory, it can be shown that standard
stabilized virtual elements are optimally convergent.
Section~\ref{sec:stabfree} will then discuss the convergence when the
stabilization parameter is set to zero.
Finally, Section~\ref{sec:numerical_experiments} reports on several numerical
experiments which confirm the theoretical results and demonstrate that
parameter free schemes are optimal in several circumstances.

%===============================================================================
%===============================================================================

\section{Function spaces and preliminaries} % (fold)
\label{sec:function_spaces}
Throughout our paper, $\Omega$ will be a simply connected polygonal domain in
$\RE^2$. 
We begin by defining the functional framework and the operators that will 
be explicitly utilized. For an integer $s\ge0$ and a generic open bounded
domain $\omega$ with Lipschitz boundary, we denote by $H^s(\omega)$ the
usual Sobolev space of (possibly fractional) order $s$.
The symbols $\|\cdot\|_{s,\omega}$ and $|\cdot|_{s,\omega}$ denote the
corresponding norm and seminorm, respectively.
The reference to $\omega$ might be omitted when no confusion arises.
We use bold letters to indicate vector valued functions with their corresponding
functional spaces. For example, $\mathbf{L}^2(\omega):=[L^2(\omega)]^2$. 

We also use the convention $H^0(\omega):=L^2(\omega)$ with the corresponding
norm $\|\cdot\|_{0,\omega}$.
The $L^2$-inner product for both spaces $L^2(\omega)$ 
and $\mathbf{L}^2(\omega)$ is denoted by $(\cdot,\cdot)_{\omega}$. 
When no confusion may arise, the domain is omitted and the $L^2$-inner
product is simply denoted by $(\cdot,\cdot)$. 

We consider the divergence and gradient operators, denoted by $\Div$ and
$\Grad$ respectively, which are defined as follows:
\[
\aligned
&\Div \bfv &:=& \quad \frac{\partial v_1}{\partial x_1} 
              +\frac{\partial v_2}{\partial x_2},\\
&\Grad q   &:=& \quad\left(\frac{\partial q}{\partial x_1} ,
               \frac{\partial q}{\partial x_2}\right)^\top,
\endaligned
\]
where $\bfv = (v_1,v_2)^\top$ is a vectorfield represented as a
two-dimensional column vector, with $\top$ denoting the transpose, and $q$ a
scalar function.
Moreover, we consider the rotation ($\Rot$) and curl ($\Curl$) operators 
which are defined as
\[
\aligned
&\Rot\bfv &:=& \quad \frac{\partial v_2}{\partial x_1}
              -\frac{\partial v_1}{\partial x_2},\\
&\Curl q &:=& \quad\left(\frac{\partial q}{\partial x_2}, 
               \frac{-\partial q}{\partial x_1}\right)^\top.
\endaligned
\]

We also recall additional standard function spaces along with 
their corresponding norms as follows:
\[
\aligned
 \Hdiv &:= \{\bfv \in \Ltwovecomega: \Div \bfv \in L^2(\Omega) \}, 
 \text{ with }  \|\bfv\|^2_{\Div} := \|\bfv\|^2_{0}+\|\Div\bfv\|^2_{0}, \\
 \Hrot &:= \{\bfv \in \Ltwovecomega:\Rot \bfv \in L^2(\Omega) \}, 
 \text{ with }  \|\bfv\|^2_{\Rot} := \|\bfv\|^2_{0}+\|\Rot\bfv\|^2_{0}.\\
  \endaligned
  \]
%===============================================================================

Let $\bfv$ and $q$ be sufficiently smooth and let $\bfn$ and $\bft$ be the
outer unit normal and counterclockwise unit tangent vectors to $\Omega$,
respectively, then the integration by parts for both divergence and rotation
operators reads
\[
\aligned
(\Div \bfv, q) &=& -(\bfv, \Grad q) 
+ (\bfv \cdot \bfn,q)_{\partial\Omega}, \\
(\Rot\bfv, q) &=& (\bfv, \Curl q)
 + (\bfv \cdot \bft,q)_{\partial\Omega}. 
\endaligned
\]

%===============================================================================
%
%However, in our case we deal with the space $ \Hrot $ which satisfies the 
%following integration by part property
%\begin{equation}
%\label{eq:IBP_property}
%  \Hrot =\{ \bfv\in\Hrot: \text{ with } (\Rot \bfv , q) = (\bfv, \Curl q), 
%  \quad \forall q \in \Honezero \}. % this is correct since its for all q in h1_zero
%\end{equation}
%==============================================================================

Specifically, we will deal with $\Hone$ and $\Honezero$ equipped with norm 
$\|\cdot\|_1$, where $\Honezero$ is defined, in the sense of the trace operator
$\gamma$ on the boundary $\partial\Omega$, by
\[
\Honezero := \{v\in\Hone: \gamma (v)=0\}.
\]
We recall that for $q\in\Honezero$ the above integration by parts simplifies
to
\[
\aligned
&(\Div \bfv, q) = -(\bfv, \Grad q),\\
&(\Rot \bfv , q) = (\bfv, \Curl q).
\endaligned
\]

Moreover, we define two subspaces of $\Hdiv$ and $\Hrot$
\[
\aligned
\Hodiv&:=\{\bfv\in\Hdiv: \bfv\cdot\bfn=0,&\text{on}& \ \partial\Omega\}, &\\
%\qquad &= \{\bfv \in \Hdiv : (\bfv ,\Grad q) = - ( \Div \bfv, q), \quad\forall q\in \Hone \}, \\ 
\Hroto &:= \{\bfv \in \Hrot:\ \Rot \bfv = 0,&\text{in}& \  \Omega\}. \\ 
\endaligned
\]
%==============================================================================
%
% Note that derivatives in both directions $x$ and $y$ in $\RE^2$, 
% (in the sense of distribution), are bounded in $L^2$ for $ \Hone$ 
% and $ \Hcurl$. Thus, $ \Hone = \Hcurl$(see~\cite{kik_discrete}) where the 
% later is defined by
% %
% \[
% \Hcurl=\{ q\in L^2(\Omega), \ \Curl q \in  \Ltwovecomega \}.
% \]
%===============================================================================
%

Finally, we recall the following compactness property
(see~\cite{kik_weak,SDCP}).
\begin{lemma}
\label{lem:Hodiv_compactness_L2}
If $\Omega$ is a polygonal domain, then there exists $s\in(1/2,1)$, such that
the subspace $\Hodiv\cap\Hroto$ is contained in $\Hsbold$ which is compactly
embedded into $\Ltwovecomega$, that is
\[
\Hodiv \cap\Hroto \subset \Hsbold \subset\subset \Ltwovecomega, \quad
s>\frac{1}{2}.
\]
 \end{lemma} 
 % end section

%===============================================================================
%===============================================================================

\section{The continuous problem} % (fold)
\label{sec:the_continues_problem}
In this section, we recall the continuous strong formulation of a model
describing the free vibrations of an acoustic fluid within a bounded rigid
cavity in $\RE^2$. 
We derive several variational formulations and discuss the connections between
these formulations and the original problem. Additionally, we obtain a mixed
formulation that we are going to use for the analysis of the problem, and we
prove its equivalence to the original variational formulation.
%===============================================================================
%===============================================================================

\subsection{Problem setting and its variational formulation}

%===============================================================================
%===============================================================================

We consider the following boundary value source problem. 
Given $\bff\in \Ltwovecomega$, find $\bfu$ such that:
\begin{equation}\label{eq:source_strong_form}
\left\{
\aligned
&-\Grad \Div \bfu = \bff, &\quad \text{ in }  \Omega, \\
&\Rot \bfu = 0, &\quad \text{ in }  \Omega, \\
&\bfu \cdot \bfn = 0, &\quad \text{ on } \partial\Omega,
\endaligned
\right.
\end{equation}
where $\bfu$ is the fluid displacement and $\bfn$ is the outer unit normal
vector to the boundary $ \partial \Omega$. 

Our main focus is to study the eigenvalue problem associated 
with~\eqref{eq:source_strong_form}:
find the eigenpair $(\lambda,\bfu)$ with $\bfu\neq \bfzero$ such that: 
\begin{equation}\label{eq:eig_strong_form}
\left\{
\aligned
&-\Grad \Div \bfu = \lambda \bfu, &\quad \text{ in }  \Omega, \\
&\Rot \bfu = 0, &\quad \text{ in }  \Omega, \\
&\bfu \cdot \bfn = 0, &\quad \text{ on } \partial\Omega.
\endaligned
\right.
\end{equation}
In this formulation we consider isotropic materials and we set all the
involved coefficients equal to one. Our analysis extends naturally to the
general situation as long as the solution is regular enough. For more details
on the derivation of the model, the interested reader can refer
to~\cite{Rodolfo} and the references therein.
%===============================================================================
%

The most natural variational form of Problem~\eqref{eq:eig_strong_form} reads
as follows: find $(\lambda,\bfu)\in\RE\times\Hodiv\cap\Hroto$ with
$\bfu\neq\bfzero$ such that 
\begin{equation}
\label{eq:pbWithRot_constrain}
(\Div\bfu,\Div\bfv)=\lambda (\bfu,\bfv) \quad \forall \bfv \in \Hodiv \cap \Hroto.\\
\end{equation} 
Since our problem is symmetric, we can confine ourselves to real eigenvalues.
Moreover, we recall that the solution operator associated
to~\eqref{eq:pbWithRot_constrain} is compact in $\Ltwovecomega$ thanks to
Lemma~\ref{lem:Hodiv_compactness_L2}. As a consequence, 
Problem~\eqref{eq:pbWithRot_constrain} admits a countable set of eigenvalues,
which can be ordered in a non decreasing divergent sequence (multiple
eigenvalues are repeated accordingly to their multiplicity)
\[
0<\lambda_1\le\lambda_2\le \cdots\le\lambda_n\le\cdots,
\]
with associated eigenfunctions chosen such that
\[
\aligned
&(\bfu_i,\bfu_j)=0, \quad (\Div\bfu_i,\Div\bfu_j)=0, \text{ if }i\ne j\\
&\|\bfu_i\|_0=1, \quad \|\Div\bfu_i\|^2_0=\lambda_i.
\endaligned
\]

%===============================================================================

It is well known that this formulation is not good for numerical 
approximation, since it requires to construct a
finite dimensional subspace of $\Hodiv \cap \Hroto$. In particular,
the rotation free constraint needs to be imposed 
exactly for a conforming approximation of~\eqref{eq:pbWithRot_constrain}.
Thus, in general, another formulation is considered by looking for
eigensolutions of~\eqref{eq:pbWithRot_constrain} in the space
$\V := \Hodiv$.
However, this implies that the zero frequency $\lambda = 0$, associated with
the infinite dimensional eigenspace $\Curl(\Honezero)$, is added to the
spectrum. 
The way to tackle this problem is to discard the zero eigenvalue after
discretizing. 

Therefore, the problem we investigate reads:
find $(\lambda,\bfu)\in \RE\times \V$ with $\bfu\neq \bfzero$ such that 
\begin{equation}
\left\{
\aligned
&(\Div\bfu,\Div\bfv)=\lambda(\bfu,\bfv) \quad \forall\bfv\in \V,\\
&\lambda\ne0.
\endaligned
\right.
\label{eq:pbcont}
\end{equation}

%===============================================================================

\begin{remark}\label{rem:spurious_modes}
It is well known that the Galerkin  discretization of~\eqref{eq:pbcont}
requires a careful choice of finite dimensional supspaces.
Indeed, the discrete spectrum can be characterized by the presence of two
kinds of spurious modes.
One might be coming from the zero frequency polluting the whole spectrum and
the other might originate from the numerical method itself,
see~\cite{bbg1,bfgp} and~\cite{DBNote_DeRham}.
\end{remark}
%===============================================================================

The kernel of the divergence operator, defined by
$$
\K^{\Div} := \{\bfv \in \V : \Div \bfv = 0\},
$$
has a crucial role in our formulation.
The following lemma characterizes the decomposition of 
the space $\V$ that will be used later on, see~\cite[Lemma~2]{Rodolfo}
and~\cite[Chapter I]{GR}.
%
%===============================================================================
\begin{lemma}
\label{lem:lem2_rodolfo}
	Let $\G:=\{ \Grad q : q \in \Hone \}$. Then
\[
	\V = \K^{\Div} \oplus (\G \cap \V) 
\]
is an orthogonal decomposition in both $\Ltwovecomega$ and $\Hdiv$.
Moreover, there exists $s\in (\frac{1}{2},1]$ such that, 
for all $\bfv \in \V$, $\bfv = \Grad q +\bfpsi $, with
$\Grad q \in (\G \cap \V)$ and $\bfpsi \in \K^{\Div}$, and
\begin{equation}
 	\Grad q \in \Hsbold\quad \text{ with }\quad \|\Grad q \|_{s} \le C \|\Div \bfv \|_{0}. 
\end{equation}
\end{lemma}

%===============================================================================
%===============================================================================

\subsection{Mixed formulation}

%===============================================================================
%===============================================================================

Another way to deal with the rotation free constraint is by considering a
\textit{mixed formulation} obtained by adding a Lagrange multiplier associated
with the constraint.
We consider the mixed formulation only for the analysis of~\eqref{eq:pbcont},
while the numerical discretization is performed utilizing the standard formulation.

Let us set $ Q :=\Honezero$, then the mixed formulation of
Problem~\eqref{eq:pbcont} reads:
find $(\lambda,\bfu)\in\RE\times\V$ with $\bfu\ne\textbf{0}$
such that for some $p\in Q$ it holds
\begin{equation}
\left\{
\aligned
&(\Div\bfu,\Div\bfv)+(\bfv,\Curl p)=\lambda(\bfu,\bfv) && \forall\bfv\in\V,\\
&(\bfu,\Curl q)=0 && \forall q\in Q.
\endaligned
\right.
\label{eq:kikcont}
\end{equation}

This is the \emph{rotated} version of the so called Kikuchi formulation used
for the approximation of the Maxwell eigenvalue problem~\cite{kik}.

%=============================================================================== 
Problems~\eqref{eq:pbcont} and~\eqref{eq:kikcont} are equivalent as
stated in the following proposition.

\begin{proposition}
Let $(\lambda,\bfu)$ be a solution of~\eqref{eq:pbcont}, then $(\lambda,\bfu)$
solves~\eqref{eq:kikcont} with $p=0$. Conversely, if $(\lambda,\bfu)$
solves~\eqref{eq:kikcont} for some $p$, then $(\lambda,\bfu)$ is also a
solution of~\eqref{eq:pbcont}.
\label{pr:cont}
\end{proposition}

%===============================================================================
\begin{proof}

Let $(\lambda,\bfu)$ be a solution of~\eqref{eq:pbcont}. Then $\lambda\ne0$
implies that $(\bfu,\Curl q)=0$ for all $q\in Q$ (take
$\bfv=\Curl q\in\V$ in~\eqref{eq:pbcont}). 
Hence $(\lambda,\bfu)$ solves~\eqref{eq:kikcont} with $p=0$.

Conversely, if $(\lambda,\bfu)$ solves~\eqref{eq:kikcont}, then necessarily
$p=0$. Indeed, we can take $\bfv=\Curl p\in\V$ in~\eqref{eq:kikcont} and
it follows $\|\Curl p\|_{0}^2=\lambda(\bfu,\Curl p)=0$, that is $p=0$ due to the
boundary conditions. 
It remains to show that $\lambda$ is different from zero.
If not, we would have from the first equation in~\eqref{eq:kikcont} that
$\Div\bfu=0$ which, together with $\Rot\bfu=0$ (consequence of the second
equation in\eqref{eq:kikcont}) implies $\bfu=0$ that is not allowed.
\end{proof}

We recall the de Rham complex related to the mixed formulation we
have just presented, which in two dimensions reads as follows
\begin{equation}
\minCDarrowwidth25pt
\begin{CD}
0@>>>\Honezero@>\Curl>>\V@>\Div>>L^2_0(\Omega)@>>>0.
\end{CD}
\label{eq:derham}
\end{equation}
This horizontal line is exact in the sense that the range of an operator 
in the sequence coincide with the kernel of the next one.  This means 
that the range of the $\Curl$ operator is equal to the kernel of the 
divergence operator. 
The first zero means that the $\Curl$ is injective 
and the last one means the $\Div$ operator is surjective.

%===============================================================================
%
In the analysis of eigenvalue problems it is useful to introduce the solution
operator, which is in general defined by means of the associated source
problem that in our case reads:
find $(\bfu,p)\in \V\times Q$ such that
%%%%%%%%%%%%%%
\begin{equation}
\left\{
\aligned
&(\Div\bfu,\Div\bfv)+(\bfv,\Curl p)=(\bff,\bfv) &&\forall\bfv\in\V,\\
&(\bfu,\Curl q)=0 &&\forall q\in Q,
\endaligned
\right.
\label{eq:kikcont_source}
\end{equation}
where the source $\bff$ replaces $\lambda \bfu$ in~\eqref{eq:kikcont}. 
Following the convention in~\cite{BBG2}, this mixed formulation is a problem of
the type
%%%%%%%%%%%%%%%%%%%%%%%%%%%%%%%%
$\begin{pmatrix}
\bff\\
0 
\end{pmatrix}$
%%%%%%%%%%%%%%%%%%%%%%%%%%%%%%%%%%%5
with $\bff\in \Ltwovecomega$ being the source. 

Then, we define the solution operator $T:\Ltwovecomega\to\Ltwovecomega$ as
follows: 
\begin{equation}
\label{eq:defT}
\text{for any }\bff\in\Ltwovecomega,\ T\bff=\bfu,
\end{equation}
with $\bfu$ being the first component of the solution of~\eqref{eq:kikcont_source}.

The operator $T$ is compact, since the first component $\bfu$ of the solution
to~\eqref{eq:kikcont_source} belongs to $\V\cap\Hroto$, which is compactly
embedded in $\Ltwovecomega$ thanks to Lemma~\ref{lem:Hodiv_compactness_L2}. 

In the rest of this section we discuss existence and uniqueness of the solution
of~\eqref{eq:kikcont_source}.

Since the second equation in~\eqref{eq:kikcont_source} is equivalent to
requiring that the solution $\bfu\in\V$ is rotation free, we define  
the kernel associated to the rotation operator as
\[
\Kerrot := \{\bfv \in \V: (\bfv,\Curl q) =  0  \quad \forall q \in Q \}.
\]

It is well known that two conditions are necessary and sufficient for the
solvability of a mixed system, namely the ellipticity in the kernel and the
inf-sup condition~\cite{bbf}. In our case, these are the ellipticity of the
bilinear form $(\Div\bfu,\Div\bfv)$ in the kernel of the rot operator
$\Kerrot$ and the inf-sup condition for the bilinear form $(\bfv,\Curl q)$. 

The bilinear from $(\Div \bfu, \Div \bfv)$ is coercive in $\Kerrot$. 
This can be easily seen as a consequence of the Friedrichs inequality. 
Indeed, 
\[
(\Div \bfv,\Div \bfv) = \|\Div \bfv\|_0^2 \ge \ C_{F} \|\bfv\|_0^2 
\quad\forall \bfv \in \Kerrot.
\]
Hence, there exists an ellipticity constant $\alpha = \frac{1}{2}\min(C_{F},1)$
such that,
\[
(\Div \bfv,\Div \bfv) \ge \alpha \|\bfv\|_{\Div}^2 
\quad\forall \bfv \in \Kerrot.
\]
By the Poincar\'e inequality 
\begin{equation}
\label{eq:Poincare_curl}
\|\Curl q\|_0 \ge C \|q\|_1  \quad \forall q \in Q,
\end{equation}
and the definition of the bilinear form, given $q\in Q$, we can choose
$\bfv = \Curl q \in \V $ and with the use of~\eqref{eq:Poincare_curl}, we get
\[
\sup_{\bfv \in \V} \frac{(\bfv, \Curl q)}{\|\bfv\|_{\Div}} 
\ge
\|\Curl q\|_0
\ge 
C \|q\|_1 \quad \forall q\in Q,
\]
which gives
\[
\inf_{q\in Q}\sup_{\bfv\in \V} \frac{(\bfv, \Curl q)}{\|\bfv\|_{\Div} \|q\|_1}
\ge \beta, 
\]
where $\beta=C$ is the inf-sup constant.

%===============================================================================

\section{The virtual element discretization} % (fold)
\label{sec:the_virtual_element_discretization}
  
%===============================================================================
%===============================================================================

In this section we briefly define the virtual element space introduced
in~\cite{Rodolfo}. First we recall the basic assumptions on the mesh, then we
describe the VEM space and state the discretized version of
Problems~\eqref{eq:pbcont} and~\eqref{eq:kikcont}.

Let $\{\T_h\}$ be a family of finite decomposition of the domain $\Omega$ into 
non-overlapping polygonal elements $E$.
We denote by $h_E$ the diameter of 
$E$, by $h_e$ the length of the edge $e\subset\partial E$ and by $h$ the mesh
size, that is the maximum of $h_E$ for $E\in\T_h$.

We suppose that for all meshes there exist a constant $C_\tau>0$ such that for
every element $E\in\T_h$ and every $\T_h$, the following standard assumptions
hold true: each element is star-shaped with respect to a disk of radius
greater than $C_\tau h_E$; the ratio between the shortest edge $e$ of $E$ and
the diameter $h_E$ is greater than $C_\tau$, that is $h_e\ge C_\tau h_E$.

Let $\omega$ be a subset of $\RE^2$, for a non-negative integer $k\geq0$,
we denote by $\Po_{k}(\omega)$ the space of polynomials of degree up to $k$
in $\omega$.
We consider the following local finite dimensional space in $E$
introduced in~\cite{Rodolfo} and inspired by~\cite[Remark~6.3]{basic_mixed}:
\begin{equation}
\label{def:VEM_E_space}
\aligned
\VE :=&\{\bfv_h\in\HdivE\cap\HrotE:\
\bfv_h\cdot\bfn\in \Po_k(e)\ \forall e\in\partial E,\\
&\qquad\qquad\qquad\Div\bfv_h\in\Po_k(E),\ 
\Rot\bfv_h=0 \ \text{ in }E\}.\\
\endaligned
\end{equation}
A function $\bfv_h\in\VE$ is uniquely determined by the following degrees of
freedom:
\[
\aligned
& \int_e(\bfv_h\cdot\bfn)q\,dS
&&\quad\forall q\in\Po_k(e),\quad\forall e\subset\partial E,\\
& \int_E \bfv_h\cdot\Grad q\,d\bfx
&&\quad\forall q\in\Po_k(E)\slash\RE.
\endaligned
\]
The global virtual element space is obtained by ensuring the 
continuity of the normal components of the local spaces, that is
\begin{equation}
\label{eq:VEM-space}
\Vh :=\{\bfv_h\in\V:\bfv_h|_E\in\VE\ \forall E\in\T_h\}.
\end{equation} 
%
%===============================================================

In view of the discrete counterparts of Problems~\eqref{eq:pbcont}
and~\eqref{eq:kikcont}, we define the discrete version of the bilinear forms
given in~\eqref{eq:pbcont}. 

We observe that the left hand side of~\eqref{eq:pbcont} can be computed exactly
since $\Div\bfv_h$ is a polynomial of degree $k$ in each element
for $\bfv_h \in\VE$. Hence we do not need to introduce any projection
operator, nor to stabilize the left hand side of our problem.
On the other hand, the right hand side contains purely virtual components and
needs to be carefully dealt. 
A standard ingredient for the numerical approximation is the construction of a
discrete bilinear form $\bho(\cdot,\cdot)$ which replaces the $\Ltwovecomega$
scalar product.
As is usual in the framework of the virtual element method,
a suitable projection operator is introduced, which allows us to compute the
discrete $\bho$. Since the elements of our local space are rotation free, they
can be represented as gradients. Therefore, \cite{Rodolfo} introduced
the operator $\PE$ on each element $E$ as the $\Ltwovec$ orthogonal
projection operator onto the space of gradients of polynomials of degree $k+1$,
that is:
\begin{equation}
\label{eq:PiE}
\aligned
&\PE: \Ltwovec \rightarrow \Grad(\Po_{k+1}(E))\subset\VE, \\
&(\PE\bfv - \bfv, \Grad q)_E=0 \quad\forall q\in\Po_{k+1}(E).
\endaligned
\end{equation} 
Then the local discrete bilinear form $\bE(\cdot,\cdot)$ on each $E$, is
defined by
\begin{equation}
\label{eq:localBE}
\bE(\bfu_h,\bfv_h):=(\PE\bfu_h,\PE\bfv_h)_E \quad\forall\bfu_h,\bfv_h\in\VE,
\end{equation}
and in a natural way, we sum up the local discrete bilinear forms to
obtain the global form
\begin{equation}
\label{eq:global_bh}
\bho(\bfu_h,\bfv_h):=\sum_{E\in\T_h}\bE(\bfu_h,\bfv_h)
\quad\forall\bfu_h,\bfv_h\in\Vh.
\end{equation}
We observe that $\bho(\bfv_h,\bfv_h)\ge 0$ for all $\bfv_h\in\Vh$, so that we
can associate it with the following seminorm which will be useful in our
analysis:
\[
|\bfv_h|_{h,0}^2:=\bho(\bfv_h,\bfv_h)=\sum_{E\in\T_h}\bE(\bfv_h,\bfv_h)
            = \sum_{E\in\T_h}\|\PE\bfv_h\|_{0,E}^2.
\]
%===============================================================================

\subsection{The discrete variational formulation}
\label{sec:pencil}
The discretization of~\eqref{eq:pbcont} consists in finding
$(\lambda_h,\bfu_h)\in \RE\times \Vh$ with $\bfu_h\ne \bfzero$ such that 
\begin{equation}
\left\{
\aligned
&(\Div\bfu_h,\Div\bfv_h)=\lambda_h\bho(\bfu_h,\bfv_h) &&\forall\bfv_h\in\Vh,\\
&\lambda_h\ne0.
\endaligned
\right.
\label{eq:pbdisc0}
\end{equation}
The algebraic system associated with the discrete eigenvalue problem has the
form
\[
\m{A}\m{x}=\lambda\m{B}\m{x}
\]
with $\m{A}$ and $\m{B}$ symmetric and positive semidefinite matrices of
dimension $N_h=\dim{\Vh}$. 
Notice that this algebraic eigenvalue problem is parameter free.
This is due to the fact that the bilinear form $(\Div\bfu_h,\Div\bfv_h)$ can be
computed exactly using the degrees of freedom 
and that the bilinear form $\bho(\bfu_h,\bfv_h)$ which corresponds
to the matrix $\m{B}$ does not depend on any parameter.

In practice, there might exist a $\bfw_h\in\Vh$ so that $\bho(\bfw_h,\bfw_h)=0$
with $\bfw_h\neq \bfzero$ (note that in this case we have also
$\bho(\bfw_h,\bfv_h)=0$ for all $\bfv_h\in\Vh$). This is not an issue for the
problem defined in~\eqref{eq:pbdisc0} unless it happened that $\Div\bfw_h=0$
as well.
In such case the eigenvalue $\lambda_h$ would not be determined by the
equation and we would be in presence of a singular pencil.

%In order to avoid this situation, 
%resulting in a degenerate problem, we reduce
%the discrete space $\Vh$ by eliminating the troublesome elements. Formally, 
%
In order to better describe this issue, we can introduce the kernels of the
matrices $\m{A}$ and $\m{B}$ which in this case are defined as follows:
\[
\aligned
&\kerdiv := \{\bfv_h \in \Vh : \Div \bfv_h = 0 \},\\
&{\kerbh}_{,0} :=\{\bfv_h\in\Vh:\bho(\bfv_h,\bfv_h)=0\},
\endaligned
\]
and their intersection
\[
\inter := \kerdiv\cap{\kerbh}_{,0}.
\]
In order to avoid degeneracy of eigenvalues it is needed that
\begin{equation}
\label{eq:inter_kernel}
\inter=\{\m{0}\}.
\end{equation}
Due to the definition of our discrete space $\Vh$ this can be achieved if, for
example, ${\kerbh}_{,0}=\{\m{0}\}$.
We shall prove that in some cases, this condition is actually satisfied with
the definition of $\bho^E$ given above. 
Otherwise, generally, one way to obtain it
is to use a stabilized bilinear form as it is custom in the virtual element
method. 

For $E\in\T_h$, let $\SE(\cdot,\cdot)$ be any symmetric positive definite
bilinear form such that there exists positive constants $\underline{c}$ and
$\overline{c}$ such that
\begin{equation}
\label{eq:bound_SE}
\underline{c}\|\bfv_h\|_{0,E}^2\le \SE(\bfv_h,\bfv_h)\le
\overline{c}\|\bfv_h\|_{0,E}^2\qquad\forall\bfv_h\in\VE.
\end{equation}
Then we define the local stabilized bilinear form for all $\bfu_h,\bfv_h\in\VE$
as follows
\begin{equation}
\label{eq:stab_bh_loc}
\bhs^E(\bfu_h,\bfv_h)
:=\bho^E(\bfu_h,\bfv_h)+\SE(\bfu_h-\PE\bfu_h,\bfv_h-\PE\bfv_h),
\end{equation}
and the global stabilized bilinear form reads
\begin{equation}
\label{eq:stab_bh}
\bhs(\bfu_h,\bfv_h):=\sum_{E\in\T_h} \bhs^E(\bfu_h,\bfv_h) 
\qquad\forall\bfu_h,\bfv_h\in\Vh.
\end{equation}
With the above definitions, it follows that $\bhs(\cdot,\cdot)$ is equivalent
to the $\mathbf{L}^2$-norm (see~\cite{Rodolfo}), indeed there exist two
positive constants $\underline{\beta}$ and $\overline{\beta}$ such that
\[
\underline{\beta}\|\bfv_h\|_{0,E}^2\le \bhs(\bfv_h,\bfv_h)\le
\overline{\beta}\|\bfv_h\|_{0,E}^2\qquad\forall\bfv_h\in\VE.
\]

\begin{remark}
Another way to circumvent the degeneracy of eigenvalues consists in discarding
from the space $\Vh$ the elements in $\inter$. Namely, 
let us denote by $\ell$ the dimension of $\inter$ and form a basis in $\Vh$
consisting of $\dim(\Vh)-\ell$ elements not in $\inter$ and $\ell$ elements in
$\inter$. We denote by $\Vhh$ the space generated by the $\dim(\Vh)-\ell$
elements not in $\inter$.
In practice, nobody would like to implement the space $\Vhh$, 
however, one can use $\Vh$ and perhaps discard the spurious modes arising
from the degeneracy. 
\end{remark}

For ease of notation, from now on we denote by $\bh$ either the bilinear form
$\bho$ defined in~\eqref{eq:global_bh} or the stabilized one $\bhs$ given
in~\eqref{eq:stab_bh}.
The $0$ in $\bho$ and the $s$ in $\bhs$ mean non-stabilized and stabilized,
respectively.
The specific choice will be made precise when needed.
Moreover, we denote by $|\cdot|_h$ the associated discrete seminorm that is
\begin{equation}
\label{eq:seminormh}
|\bfv_h|^2_h:=\bh(\bfv_h,\bfv_h).
\end{equation}
%
%\marginnote{ We have defined the semin-not for$\bho$ and now the same for $b_h$}
Analogously, we will extend the use of the notation of its kernel as
\[
\kerbh:=\{\bfv_h\in\Vh:\bh(\bfv_h,\bfv_h)=0\}.
\]

Hence, the discrete problem we analyze is associated with the space $\V_h$
defined in~\eqref{eq:VEM-space} and reads as follows: find
$(\lambda_h, \bfu_h)\in\RE\times\Vh$ with $\bfu_h\ne\bf0$ such that
\begin{equation}
\left\{
\aligned
&(\Div\bfu_h,\Div\bfv_h)=\lambda_h b_h(\bfu_h,\bfv_h) &&\forall\bfv_h\in\Vh,\\
&\lambda_h\ne0.
\endaligned
\right.
\label{eq:pbdisc}
\end{equation}

Problem~\eqref{eq:pbdisc} admits exactly $N_h=\dim(\Vh)-\dim(\kerdiv)$
discrete eigenvalues
$\lambda_{i,h}$, $i=1,\dots,N_h$ if $\kerbh=\{\m{0}\}$, with discrete
eigenfunctions satisfying the following orthogonality properties
\[
\aligned
&\bh(\bfu_{i,h},\bfu_{j,h})=0,\quad (\Div\bfu_{i,h},\Div\bfu_{j,h})=0,
\text{ if }i\ne j,\\
&\bh(\bfu_{i,h},\bfu_{i,h})=1,\quad \|\Div\bfu_{i,h}\|^2_0=\lambda_{i,h}.
\endaligned
\]
Notice that in this case the eigenfunctions are orthogonal with respect to the
mesh dependent form $\bh$ instead of the more standard $\mathbf{L}^2$ scalar
product.

%===================================================
\subsection{Discrete mixed formulation} 
In order to define the discrete counterpart of~\eqref{eq:kikcont} that we shall
use for the analysis, we introduce a finite dimensional subspace $\Qh$ of
$Q=\Honezero$.
Let $\Vh\subset\V=\Hodiv$ be the space defined above and let $\Qh\subset Q$
be any space such that
\[
\Curl\Qh\subset\Vh.
\]
Given the discrete space $\Vh$, the space $\Qh$ can be constructed as follows.
We consider the kernel of the div operator in $\Vh$. This is what we called
$\kerdiv$.  Each element of $\kerdiv\subset\Vh$ can be represented as the
$\Curl$ of a unique function in $Q$ thanks to the boundary conditions of $\V$.
We define $\Qh$ as the subspace of $Q$ containing all such functions, that is
\begin{equation}
\label{eq:defQh}
  \Curl Q_h = \kerdiv\subset\Vh,
\end{equation}
which is the compatibility assumption on the discrete spaces $\Qh$ and $\Vh$. 

In particular, in our framework, we have the following diagram
\[
\begin{CD}
Q@>\Curl>>\V\\
@VVV @VVV\\
\Qh@>\Curl>>\Vh
\end{CD}
\]

\begin{remark}
  The space $Q_h$ is only used for the analysis, 
  and will not be implemented in our numerical experiments. 
\end{remark}

The discrete mixed formulation of~\eqref{eq:kikcont} then reads:
find $(\lambda_h, \bfu_h)\in\RE\times\Vh$ with $\bfu_h\ne \bfzero$ 
such that for some $p_h\in\Qh$ it holds
\begin{equation}
\left\{
\aligned
&(\Div\bfu_h,\Div\bfv_h)+\bh(\bfv_h,\Curl p_h)=\lambda_h\bh(\bfu_h,\bfv_h) &&
\forall\bfv_h\in\Vh,\\
&\bh(\bfu_h,\Curl q_h)=0 &&\forall q_h\in\Qh.
\endaligned
\right.
\label{eq:kikdisc}
\end{equation}
The rotation free constraint is now substituted by the second equation
in~\eqref{eq:kikdisc} and we define the associated discrete kernel as
\[
\Kbhrot := \{\bfv_h \in \Vh : b_h(\bfv_h,\Curl q_h) =  0
\quad \forall q_h \in Q_h \}. 
\] 
Note that in general $\Kbhrot\not\subset \Kerrot$.

The next proposition shows the equivalence between 
the discrete mixed formulation~\eqref{eq:kikdisc} and 
the discrete variational formulation~\eqref{eq:pbdisc}.
%=============================================================
\begin{proposition}
Let us assume that $\kerbh=\{\m{0}\}$.
Let the pair $(\lambda_h,\bfu_h)$ be an eigensolution of~\eqref{eq:pbdisc}, then
$(\lambda_h,\bfu_h)$ solves~\eqref{eq:kikdisc} with $p_h=0$.
Conversely, let $(\lambda_h,\bfu_h)$ solve~\eqref{eq:kikdisc} for
some $p_h$, then $(\lambda_h,\bfu_h)$ solves~\eqref{eq:pbdisc}.

\end{proposition}

%=============================================================

\begin{proof}

As in the proof of Proposition~\ref{pr:cont}, if $(\lambda_h,\bfu_h)$ is a
solution of~\eqref{eq:pbdisc}, then by choosing $\bfv_h = \Curl q_h$ we 
have $b_h(\bfu_h,\Curl q_h)=0$ for all
$q_h\in\Qh$ from $\lambda_h\ne0$ and $\Curl\Qh\subset\Vh$. Hence
$(\lambda_h,\bfu_h)$ solves~\eqref{eq:kikdisc} for $p_h=0$.

Conversely, let $(\lambda_h,\bfu_h)$ be a solution of~\eqref{eq:kikdisc} for 
some $p_h\in\Qh$. 
Taking $\bfv_h=\Curl p_h$ in the first equation of~\eqref{eq:kikdisc}, we have,
recalling the definition of the seminorm associated to $\bh$,
\[
|\Curl p_h|_h^2=\bh(\Curl p_h,\Curl p_h)=\lambda_h\bh(\bfu_h,\Curl p_h)=0.
\]
It follows that
\[
|b_h(\bfv_h,\Curl p_h)|\le|\bfv_h|_h|\Curl p_h|_h=0 \quad\forall\bfv_h\in\Vh.
\]
Hence, it remains to show that $\lambda_h$ cannot be zero. 
By contradiction, let $\lambda_h=0$, then from the first equation
in~\eqref{eq:kikdisc} if follows that $\Div\bfu_h=0$, that is,
$\bfu_h\in\kerdiv$. 
Moreover, from~\eqref{eq:defQh}, there exists $q_h\in\Qh$ such that
$\Curl q_h=\bfu_h$ and thus, using the second equation in~\eqref{eq:kikdisc},
we get
\[
\bh(\bfu_h,\bfu_h) = | \bfu_h |^2_h = 0,
\]
that is, $\bfu_h\in\kerbh$, so that $\bfu_h=0$,
which contradicts the fact that it is an eigenfunction of~\eqref{eq:kikdisc}.
\end{proof}
%=============================================================

We end this section by introducing the approximation of the source
problem~\eqref{eq:kikcont_source}: given $\bff\in \Ltwovecomega$, find
$(\bfu_h, p_h)\in\V_h \times Q_h$ such that
\begin{equation}
\left\{
\aligned
&(\Div\bfu_h,\Div\bfv_h)+\bh(\bfv_h,\Curl p_h)=\bh(\bff,\bfv_h) &&
\forall\bfv_h\in\Vh,\\
&\bh(\bfu_h,\Curl q_h)=0 &&\forall q_h\in\Qh.
\endaligned
\right.
\label{eq:kikdisc_source}
\end{equation}
Then the discrete solution operator is given by
\begin{equation}
\label{eq:defTh}
\aligned
&T_h:\Ltwovecomega\to\Ltwovecomega,\\ 
&T_h\bff=\bfu_h\in\Vh,\\
\endaligned
\end{equation}
with $\bfu_h$ being the first component of the solution
of~\eqref{eq:kikdisc_source}.
%=============================================================
%=============================================================
\section{Spectral approximation and convergence analysis}
\label{sec:abstract_theory}

In this section we discuss the spectral approximation for the problem under
consideration. 
In particular, we analyze the convergence of the spectrum using the mixed
formulation presented previously. To this aim,
we are going to use the theory developed in~\cite{SDCP}.

\subsection{Approximation properties of the VEM space and other preliminary
results}
We start by recalling and proving some approximation properties for the
discrete space $\Vh$. The first one corresponds to~\cite[Lemma~8]{Rodolfo}
\begin{lemma}
\label{lem:estPE}
There exists a constant $C>0$ such that for $p\in H^{1+s}(\Omega)$ with 
$1/2<s\le k+1$, it holds
\[
\|\Grad p-\PE(\Grad p)\|_{0,E}\le C h^s_E\|\Grad p\|_{s,E} \quad \forall E\in\T_h.
\]
\end{lemma}

The next one deals with the interpolant $\bfv^I\in\Vh$ which,
following~\cite{Rodolfo}, is defined using the degrees of freedom introduced
above. Let $\bfv\in\V$ be such that $\bfv|_E\in\mathbf{H}^s(E)$ for some
$s>1/2$ and $E\in\T_h$, so that its trace along each edge of $E$ is well
defined, then for all $E\in\T_h$, $\bfv^I\in\Vh$ satisfies
\[
\aligned
&((\bfv-\bfv^I)\cdot\bfn,q)_e=0&&\forall q\in\Po_k(e),\ \forall
e\subset\partial E \text{ with } e\not\subset\partial\Omega,\\
&(\bfv-\bfv^I,\Grad q)_E=0&&\forall  q\in\Po_k(E)\setminus\RE.
\endaligned
\] 
Let $P_k$ be the $L^2$-projection operator from $L^2(\Omega)$ onto the subspace
of $L^2(\Omega)$ consisting of piecewise discontinuous polynomials of degree
$k$ on each element $E\in\T_h$. Then for $\bfv\in\mathbf{H}^s(\Omega)$ with
$s>1/2$, we have
\[
\Div\bfv^I=P_k(\Div\bfv).
\]

For our analysis we need a suitable modification of~\cite[Lemma~7]{Rodolfo},
which relies on an estimate for the interpolation error $\|\bfv-\bfv^I\|$
in $\Vh$ (see~\cite[Lemma 6]{Rodolfo}) which is not true in general.
Since we only need to approximate gradients correctly, the following amended 
statement of~\cite[Lemma 6]{Rodolfo} can be proved.
%==============================================================
\begin{lemma}
\label{lem:interpol_est}
% This is LEmma 6 amended
Let $\bfv\in\V$ be a gradient $\bfv=\Grad p$ and satisfy the regularity
assumption $\bfv\in\mathbf{H}^s(\Omega)$ with $s>1/2$. The interpolant $\bfv^I$
satisfies for all $E\in\T_h$
\[
\aligned
&\|\bfv-\bfv^I\|_{0,E}\le Ch_E^s\|\bfv\|_{s,E}, &&1\le s\le k+1,\\
&\|\bfv-\bfv^I\|_{0,E}\le C(h_E^s\|\bfv\|_{s,E}+h_E\|\Div\bfv\|_{0,E}),
&&1/2<s\le1.
\endaligned
\]
\end{lemma}

\begin{proof}

In the proof of~\cite[Lemma~6]{Rodolfo} it is wrongly stated that for all
$\bfv_k\in[\Po_k(E)]^2$ it holds $(\bfv_k)^I=\bfv_k$ (see first line of
page~761). However, this property is true whenever $\bfv_k=\Grad p_{k+1}$ with
$p_{k+1}\in\Po_{k+1}(E)$. The rest of the proof works with no modifications.

\end{proof}
%==============================================================
We now provide a modification of the proof of~\cite[Lemma~7]{Rodolfo} to fit
our case.
There are two main differences between our situation and the one considered
in~\cite{Rodolfo}. We are looking for solutions $\bfu_h\in\Kbhrot$, that is
they satisfy the following orthogonality $\bh(\bfu_h,\Curl q_h)=0$ for all
$q_h\in\Qh$ but they might not be $L^2$-orthogonal to $\Curl(\Qh)$.
Moreover, we want also to cover the case of $\bh=\bho$, without stabilization.
%==============================================================
% This is Lemma 2 of the sketch
\begin{lemma}
\label{lem:lem7_rodo}
Let us assume that the seminorm $|\cdot|_h$ is equivalent to the $L^2(\Omega)$
norm, that is,
$\underline{c}\|\bfv\|^2_0\le|\bfv|^2_h\le \overline{c}\|\bfv\|^2_0$ for all
$\bfv\in\Ltwovecomega$.
Moreover, let $\bfv_h$ be an element of $\Kbhrot$, that is 
$\bh(\bfv_h,\Curl q_h)=0$ for all $q_h\in\Qh$.
Then a continuous Helmholtz decomposition 
$\bfv_h=\Grad p+\bfpsi$ can be written with
$p\in H^{1+s}(\Omega)$ ($1/2<s\le 1$), $\bfpsi\in\K^{\Div}$, and
\[
\aligned
&\|\Grad p\|_{s}\le C\|\Div\bfv_h\|_{0},\\
&\|\bfpsi\|_{0}\le Ch^s\|\Div\bfv_h\|_{0}.
\endaligned
\]

\label{eq:lem7_rodolfo_decomposition}
\end{lemma}

\begin{proof}
The existence of the Helmholtz decomposition and the bound for $\Grad p$
follows from~\cite[Lemma~2]{Rodolfo} and is stated in
Lemma~\ref{lem:lem2_rodolfo}. It remains to show the bound for $\bfpsi$. We
have
\[
\aligned
\underline{c}\|\bfpsi\|^2_0&\le|\bfpsi|_h^2=b_h(\Grad p-\bfv_h,\Grad p-\bfv_h)\\
&=b_h(\Grad p-\bfv_h,\Grad p-(\Grad p)^I)+
b_h(\Grad p-\bfv_h,(\Grad p)^I-\bfv_h)\\
&\le|\bfpsi|_h|\Grad p-(\Grad p)^I|_h+
|b_h(\Grad p-\bfv_h,(\Grad p)^I-\bfv_h)|\\
&\le|\bfpsi|_h\|\Grad p-(\Grad p)^I\|_{0}+
|b_h(\Grad p-\bfv_h,(\Grad p)^I-\bfv_h)|\\
&\le C|\bfpsi|_h(h^s\|\Grad p\|_{s}+h\|\Div\bfv_h\|_{0})+
|b_h(\Grad p-\bfv_h,(\Grad p)^I-\bfv_h)|,
\endaligned
\]
where we used Lemma~\ref{lem:interpol_est} and the fact that 
$|\cdot|^2_h\le\|\cdot\|_0^2$.
It remains to estimate the last term.
Since $(\Grad p)^I-\bfv_h$ belongs to $\K_h^{div}$, it follows that
by the property~\eqref{eq:defQh} of $\Qh$, 
$b_h(\bfv_h,(\Grad p)^I-\bfv_h)=0$ 
and that $(\Grad p,(\Grad p)^I-\bfv_h)=0$ since $\K_h^{div}\subset\K^{div}$,
so that
\[
\aligned
\bh&(\Grad p-\bfv_h,(\Grad p)^I-\bfv_h)=
\bh(\Grad p,(\Grad p)^I-\bfv_h)\\
&=\bh(\Grad p,(\Grad p)^I-\bfv_h)-(\Grad p,(\Grad p)^I-\bfv_h)\\
&=\sum_E\left(\bho^E(\Grad p,(\Grad p)^I-\bfv_h)
-(\Grad p,(\Grad p)^I-\bfv_h)_E\right.\\
&\qquad
+\left.\SE(\Grad p-\PE\Grad p,(\Grad p)^I-\bfv_h-\PE((\Grad p)^I-\bfv_h))
\right).\\
\endaligned
\]
We observe that the last term appears if we are considering $\bh$ to be the
stabilized form $\bhs$ defined in~\eqref{eq:stab_bh}. Therefore, 
we bound separately the terms on the two last lines of the previous identity.
Using the properties of the projector $\PE$, Lemmas~\ref{lem:estPE}
and~\ref{lem:interpol_est}, we have for each $E\in\T_h$
\begin{equation}
\label{eq:first}
\aligned
\bho^E&\big(\Grad p,(\Grad p)^I-\bfv_h\big)
-\big(\Grad p,(\Grad p)^I-\bfv_h\big)_E\\
&=\big(\PE\Grad p,\PE((\Grad p)^I-\bfv_h)\big)_E
-\big(\Grad p,(\Grad p)^I-\bfv_h\big)_E\\
&=\big(\PE\Grad p,(\Grad p)^I-\bfv_h\big)_E
-\big(\Grad p,(\Grad p)^I-\bfv_h\big)_E\\
&=\big(\PE\Grad p-\Grad p,(\Grad p)^I-\bfv_h\big)_E\\
&=\big(\PE\Grad p-\Grad p,(\Grad p)^I-\Grad p\big)_E\\
&\qquad+\big(\PE\Grad p-\Grad p,\Grad p-\bfv_h\big)_E \\
&\le Ch_E^s\|\Grad p\|_{s,E}
\left(h_E^s\|\Grad p\|_{s,E}+h_E\|\Div\bfv_h\|_{0,E}+\|\bfpsi\|_{0,E}\right).
\endaligned
\end{equation}
We estimate the term containing the stabilization form $\SE$ by
using~\eqref{eq:bound_SE}, so that, for each element $E\in\T_h$ we have
\begin{equation}
\label{eq:second}
\aligned
\SE&(\Grad p-\PE\Grad p,(\Grad p)^I-\bfv_h-\PE((\Grad p)^I-\bfv_h))\\
&\le\overline{c}\|\Grad p-\PE\Grad p\|_{0,E}
\ \|(\Grad p)^I-\bfv_h - \PE ((\Grad p)^I-\bfv_h)\|_{0,E}\\
&=\overline{c}\|\Grad p-\PE\Grad p\|_{0,E}
\ \|(\mathbb{I}-\PE)((\Grad p)^I-\Grad p-\bfpsi)\|_{0,E}\\
&\le Ch_E^s\|\Grad p\|_{s,E}
\left(h_E^s\|\Grad p\|_{s,E}+h_E \|\Div\bfv_h\|_{0,E}+\|\bfpsi\|_{0,E}\right).
\endaligned
\end{equation}
To arrive at the last inequality we used Lemmas~\ref{lem:estPE}
and~\ref{lem:interpol_est}.

Putting together all pieces we finally obtain
\[
\aligned
\|\bfpsi\|^2_0&\le
C|\bfpsi|_h(h^s\|\Grad p\|_{s}+h\|\Div\bfv_h\|_{0})\\
&\quad+Ch^s\|\Grad p\|_{s}(h^s\|\Grad p\|_{s}
+h\|\Div\bfv_h\|_{0}+\|\bfpsi\|_{0})\\
&\le C(h^s\|\Grad p\|_{s}+h\|\Div\bfv_h\|_{0})
(|\bfpsi|_h+\|\bfpsi\|_{0})\\
&\quad+Ch^s\|\Grad p\|_{s}
(h^s\|\Grad p\|_{s}+h\|\Div\bfv_h\|_{0})
\endaligned
\]
Using the fact that $|\bfpsi|_h\le \|\bfpsi\|_{0}$, we finally obtain the
required result by means of the Young inequality.
\end{proof}
\begin{remark}
We observe that the two terms on the left hand side of~\eqref{eq:first}
and~\eqref{eq:second} are bounded by the same quantities. This implies that
we have the same result if $\bh$ is either $\bho$ or $\bhs$.
\end{remark}

\subsection{Consistency} % (fold)
\label{sub:consistency}
Since the continuous and discrete formulations of the source
problems~\eqref{eq:kikcont_source} and~\eqref{eq:kikdisc_source} contain 
different right hand sides, %we estimate
%More precisely, following the analysis in~\cite[Theorem~1]{BBG2} with 
%$\bff\in \Ltwovecomega$ being the source, 
we need a uniform bound of
\[
|(\bff,\bfv_h)-b_h(\bff,\bfv_h)|  \qquad\forall\bfv_h\in\Kbhrot \subset\Vh,
\]
in terms of some $\rho(h)$ tending to zero as $h$ goes to zero times
$\|\bff\|_{0}\|\bfv_h\|_{\Div}$, in view of the application of a Strang lemma.

This is stated in the following proposition
%==============================================================

\begin{proposition}
\label{prop:consistancy}
The consistency term satisfies
\begin{equation}
\label{eq:consistancy}
  \sup_{\bfv_h\in \Kbhrot} \frac{|(\bff,\bfv_h) - b_h(\bff,\bfv_h)|}{\|\bfv_h \|_{\Div}} \leq \rho(h) \|\bff\|_{0},
\end{equation}
with $\rho(h)$ tending to zero as $h$ goes to zero.
\end{proposition}
\begin{proof}
For $\bfv_h\in\Kbhrot$ we have
  \begin{equation*}
  \begin{aligned}
   |(\bff,\bfv_h) &- b_h(\bff,\bfv_h)| 
    = \Big| \sum_{E\in\T_h} (\bff,\bfv_h)_E - b^E_h(\bff,\bfv_h)\Big|\\
    &= \Big|\sum_{E\in\T_h} (\bff,\bfv_h)_E - (\PE \bff,\PE\bfv_h)_E
-S^E(\bff-\PE \bff,\bfv_h - \PE \bfv_h)\Big|\\
&\le\Big|\sum_{E\in\T_h} (\bff, \bfv_h - \PE \bfv_h)_E\Big|
+\sum_{E\in\T_h}C\|\bff-\PE \bff\|_{0,E}\|\bfv_h - \PE \bfv_h\|_{0,E}\\
&\le C \sum_{E\in\T_h} \|\bff \|_{0,E} \  \|\bfv_h - \PE \bfv_h\|_{0,E},
  \end{aligned}
  \end{equation*}
where we used $(\PE\bff,\PE\bfv_h)_E = (\bff,\PE\bfv_h)_E$, since
$\PE\bfv_h\in\Grad(\Po_{k+1}(E))$ and the estimate
$\|\bff-\PE\bff\|_{0,E}\le 2\|\bff\|_{0,E}$.

Utilizing the continuous Helmholtz decomposition of $\bfv_h = \Grad p + \bfpsi$ 
in Lemma~\ref{lem:lem7_rodo} and the triangle inequality, we have,
\[
\aligned
|(\bff,\bfv_h) - b_h(\bff,\bfv_h)| 
&\le  \sum_{E\in\T_h} \|\bff\|_{0,E} \ \big( \|\Grad p - \PE \Grad p \|_{0,E} +
\|\bfpsi  - \PE \bfpsi\|_{0,E}\big).\\
&\le C  \|\bff\|_{0,\Omega} \big(\|\Grad p - \Pi_h \Grad p \|_{0,\Omega} +
2\|\bfpsi\|_{0,\Omega}\big).\\
\endaligned
\]   
Lemmas~\ref{lem:estPE} and~\eqref{eq:lem7_rodolfo_decomposition} finally imply
for $s>1/2$
\[
    |(\bff,\bfv_h) - b_h(\bff,\bfv_h)| 
    \le C h^s \|\bff\|_{0}  \|\Div \bfv_h\|_{0}
    \le C h^s \|\bff\|_{0}  \|\bfv_h\|_{\Div}
\]
and, hence,
\[
\sup_{\bfv_h\in \Kbhrot} 
\frac{|(\bff,\bfv_h) - b_h(\bff,\bfv_h)|}{\|\bfv_h\|_{\Div}}
\le C h^s \|\bff\|_{0},
\]
which proves~\eqref{eq:consistancy}.
\end{proof}
%=======================================================================
\subsection{Uniform convergence of $T_h$ to $T$}
In this section we discuss the conditions which imply the uniform convergence
of discrete solution operator $T_h$ defined in~\eqref{eq:defTh} to the
continuous one $T$, given in~\eqref{eq:defT}.
Since these two operators are associated to the source
problems in mixed form~\eqref{eq:kikdisc_source} and~\eqref{eq:kikcont_source},
respectively, we apply the theory developed in~\cite{BBG2}. Hence, we introduce
the necessary conditions and then we prove that they imply uniform
convergence.

First, let's define the solution spaces for Problem~\eqref{eq:kikcont_source}. 
\begin{definition}
\label{def:solution_spaces}
{\textbf{The solution spaces $V_0$ and $Q_0$.}}
Let $\V_0$ be the subspace of $\V$ and $Q_0$ be the subspace of $Q$
such that for all $\bff\in\Ltwovecomega$
the solutions $(\bfu,p)\in\V\times Q$ of Problem~\eqref{eq:kikcont_source} belong to
$\V_0\times Q_0$.

Specifically, the space of solutions $\V_0$ is defined as
\[
\V_0 := \{ \bfu \in \V: \bfu = T\bff, \text{ for some } \bff\in \Ltwovecomega\},
\]
where $T$ is the solution operator defined in~\eqref{eq:defT}.

We endow the spaces $\V_0$ and $Q_0$ with their natural norms
\begin{equation}
\aligned
\|\bfu\|_{\V_0} &:= \inf \{\|\bff\|_{0}: T\bff = \bfu\}, &\\ 
\|p\|_{Q_0} &:= \inf \{\|\bff\|_{0}: p \text{ is the second component of the
solution of~\eqref{eq:kikcont_source} with datum } \bff\}. &\\
\endaligned
\end{equation} 

Clearly, due to the second equation in~\eqref{eq:kikcont_source}, all possible 
solutions $\bfu$ are rot-free, hence  
\[
\V_0 \subset \Kerrot.
\]
Moreover, by Lemma~\ref{lem:lem2_rodolfo} we have that the solutions
$\bfu\in\V_0$ satisfy
$\bfu\in\Hsbold$ for $1/2<s\le1$ and it is easy to verify that 
$\Div \bfu\in H^1(\Omega)$. Thus, $\V_0$ is compactly embedded in
$\V$.
%\begin{equation}
%\label{eq:solution_space_V0_compact}
%\V_0 \subset \V \cap \Hroto\subset [H^s(\Omega)]^2\subset\subset \V,
%\quad(s>\frac{1}{2})
%\end{equation}

\end{definition}

We now introduce three properties that shall be used in the proof of the
uniform convergence.
\begin{definition}
{\textbf{Ellipticity in the discrete kernel (EDK).}}
\label{def:DELKR}
We say that the ellipticity in the discrete kernel of 
Problem~\eqref{eq:kikdisc} is satisfied if there exists 
a positive constant $\alpha$, independent of the  mesh size $h$, such that 
\begin{equation}
  (\Div \bfv_h, \Div \bfv_h) \geq \alpha \|\bfv_h\|^2_{\Div} \quad \forall\bfv_h \in \Kbhrot.
\end{equation}
\end{definition}

%==============================================================
\begin{definition}
{\textbf{Weak approximability of $Q_0$ (WA).}} 
We say that the solution space $Q_0$ satisfies the weak approximability
property if there exists 
$\omega_1(h)$ going to zero as the mesh size $h$ goes to zero, such that 
\begin{equation}
\underset{\bfv_h\in \Kbhrot}\sup \frac{(\bfv_h, \Curl p)}{\|\bfv_h\|_{\Div}} \leq \omega_1(h) \|p\|_{Q_0} \quad \forall p \in Q_0  .
\end{equation}
\end{definition}

%==============================================================
\begin{definition}
{\textbf{Strong approximability of $\V_0$ (SA).}}

We say that the solution space $\V_0$ satisfies the strong approximability
property  if
there exists $\omega_2(h)$ tending to zero as mesh size $h$ goes to zero, such
that for all $\bfu \in \V_0$ there exists $\bfuSA\in\Kbhrot$ satisfying
\begin{equation}
\|\bfu - \bfuSA\|_{\Div} \leq \omega_2(h)
\|\bfu\|_{V_0}.
\end{equation}   
\end{definition}
%
% We recall here the theorem related to the uniform convergence of the discrete
% solution operator $T_h$ to the continuous one $T$ as $h\to0$
% (see~\cite{BBG2}).  We provide a detailed proof in this setting taking into
% account that in our case we are in presence of a nonconforming approximation;
% indeed, we use of the bilinear form $\bh$ instead of the $L^2$-scalar product
% in the definition of our discrete problem.

Next we recall the theorem regarding the uniform convergence of the discrete solution operator $T_h$ to its continuous counterpart $T$ as $h$ approaches zero, as originally outlined in reference \cite{BBG2}. Within this context, we present an in-depth proof, considering the fact that we are dealing with a nonconforming approximation. Notably, our approach involves the utilization of the bilinear form $\bh$ instead of the $L^2$-scalar product in the formulation of our discrete problem.

\begin{theorem}
\label{th:vem_uniform_conv_sol_operator}
Assume that the EDK is satisfied together with the WA 
and SA of the spaces $Q_0$ and $\V_0$ respectively. 
Then the discrete sequence $\{T_h\}$ converges uniformly to $T$ in $\V$. 
That is, there exists $\omega_3(h)$ tending to zero as mesh size $h$ goes to zero, such that
\[
\|T\bff-T_h\bff\|_{\Div}\le \omega_3(h)\|\bff\|_{0}
          \quad \forall \bff\in \Ltwovecomega.
\]
\end{theorem}
%==============================================================

\begin{proof}
Let $(\bfu, p) \in \V \times Q$ be a solution of~\eqref{eq:kikcont} where $\bfu = T\bff$. 
  Let $(\bfu_h, p_h) \in \Vh \times Q_h$ be a solution for the discrete 
  Problem~\eqref{eq:kikdisc} with $\bfu_h=T_h \bff$ for $\bff \in \Ltwovecomega$.
   To prove uniform convergence of the operator $T_h$ to $T$, 
   we need to estimate $\|T\bff - T_h\bff\|_{\Div}$. 
   This is the same as estimating $\|\bfu - \bfu_h\|_{\Div}$. 
   Since we assume the SA of the space $\V_0$, then for all solutions $\bfu$
there exists 
   $\bfu^I_h \in \Kbhrot$ such that 
   $$
   \|\bfu-\bfu^I_h\|_{\Div}\rightarrow 0.
   $$
  We split the norm  $\|\bfu - \bfu_h\|_{\Div}$ into two pieces and use the 
  triangular inequality as follows:
   \begin{equation}
   \label{eq:error_setimate_strang}
   \aligned
     \|\bfu - \bfu_h\|_{\Div} 
              &= \|\bfu - \bfu_h \pm \bfu^I_h\|_{\Div} \\
              &\leq \|\bfu - \bfu^I_h \|_{\Div} + \|\bfu^I_h -\bfu_h\|_{\Div}. 
    \endaligned
    \end{equation}
    Looking at the second norm on the right hand side, and since the problem 
    satisfies EDK with both $\bfu^I_h$ and $ \bfu_h$ belonging to the discrete kernel $\Kbhrot$ we have,
    \[
    \aligned
     \alpha \|\bfu^I_h -\bfu_h\|^2_{\Div} 
     & \leq  \big(\Div(\bfu^I_h - \bfu_h ), \Div(\bfu^I_h - \bfu_h) \big ) \pm (\Div\bfu,\Div (\bfu^I_h - \bfu_h ))&\\ 
     &  = \big ( \Div(\bfu^I_h - \bfu ) \big ), \Div(\bfu^I_h - \bfu_h \big ) \big ) + \big (\Div(\bfu - \bfu_h ), \Div(\bfu^I_h -\bfu_h ) \big ). &\\
     \endaligned
     \]
     Now, taking the right hand side and by applying the Cauchy-Schwarz inequality to the 
     first scalar product and the error equation of~\eqref{eq:kikcont_source} to the second, we get
     \[
      \aligned             
      \alpha \|\bfu^I_h -\bfu_h\|^2_{\Div} 
     & \leq \|\bfu^I_h - \bfu \|_{\Div} \| \bfu^I_h - \bfu_h \|_{\Div} + \\
     &- \big (\bfu^I_h - \bfu_h, \Curl p \big) 
     + b_h \big (\bfu^I_h - \bfu_h, \Curl p_h \big) + \\
     & + \big ( \bff, \bfu^I_h - \bfu_h \big ) - b_h (\bff, \bfu^I_h - \bfu_h), \\
     \endaligned
     \]
     where $b_h(\bfu^I_h - \bfu_h, \Curl p_h \big)$ is equal 
     to zero since $\bfu^I_h- \bfu_h \in \Kbhrot$. 
     Therefore, we get
     \[
     \aligned
     \alpha \|\bfu^I_h -\bfu_h\|^2_{\Div} 
     & \leq \|\bfu^I_h - \bfu \|_{\Div} \| \bfu^I_h - \bfu_h \|_{\Div} +\\ 
     & +  \underset{\bfv_h\in \Kbhrot}\sup \frac{\big( \bfv_h,\Curl p\big)}{\|\bfv_h\|_{\Div}} \|\bfu^I_h - \bfu_h\|_{\Div} +\\
     & + \underset{\bfv_h \in \Kbhrot}\sup \frac{| \big (\bff,\bfv_h \big) 
     - b_h\big(\bff,\bfv_h \big) |}{\|\bfv_h\|_{\Div}}
     \|\bfu^I_h - \bfu_h\|_{\Div} .\\
     \endaligned
     \]
     Taking the common factor, we finally get,
     \[
     \aligned
          \alpha \|\bfu^I_h -\bfu_h\|^2_{\Div} 
	 &\leq \bigg (\|\bfu^I_h - \bfu \|_{\Div} + \underset{\bfv_h\in
        \Kbhrot}\sup \frac{\big( \bfv_h,\Curl p\big)}{\|\bfv_h\|_{\Div}}\\ 
     &\qquad +\underset{\bfv_h \in \Kbhrot}\sup \frac{| \big (\bff,\bfv_h \big) 
        - b_h\big(\bff,\bfv_h \big) |}{\|\bfv_h\|_{\Div}} \bigg )\|\bfu^I_h -
	    \bfu_h\|_{\Div}.
      \endaligned
      \] 
    Applying SA property for the first norm on the right hand side, WA for the
second term and the consistency~\eqref{eq:consistancy} for the last one, we get
    \[
      \alpha \|\bfu^I_h -\bfu_h\|_{\Div} \leq \bigg(\omega_2(h) \|\bfu \|_{V_0} + \omega_1(h)\|p\|_{Q_0}
      + \rho(h) \|\bff\|_{0}\bigg)
    \]
    Putting all pieces together by using the above in~\eqref{eq:error_setimate_strang}, we get

    \[
    \aligned
      \|T\bff - T_h\bff\|_{\Div} &=\|\bfu - \bfu_h \|_{\Div} 
      \le \|\bfu - \bfu^I_h \|_{\Div} + \|\bfu^I_h - \bfu_h \|_{\Div}\\
      &= \|\bfu - \bfu^I_h \|_{\Div} + \frac{1}{\alpha} \bigg(\omega_2(h) \|\bfu \|_{V_0} + \omega_1(h)\|p\|_{Q_0}
      + \rho(h) \|\bff\|_{0}\bigg)\\
      &\le \omega_2(h) \|\bfu \|_{V_0} + \frac{1}{\alpha} \bigg(\omega_2(h) \|\bfu \|_{V_0} + \omega_1(h)\|p\|_{Q_0}
      + \rho(h) \|\bff\|_{0}\bigg)\\
      %
      %&\le \bigg((1+\frac{1}{\alpha})\omega_2(h) 
      % + \frac{1}{\alpha}  \omega_1(h) 
      %+ \frac{1}{\alpha}\rho(h)\bigg) \|\bff\|_{0}\\
      &\le \omega_3(h)\|\bff\|_{0}
      \endaligned
    \]
    where $\omega_3(h) = (1+\frac{1}{\alpha})\omega_2(h) 
       + \frac{1}{\alpha}  \omega_1(h)
      + \frac{1}{\alpha}\rho(h)$. 
    Then the result follows immediately and convergence is achieved.
\end{proof}

\subsection{Discrete Compactness Property}

An important tool for the analysis of this problem is the so called
\emph{Discrete Compactness Property} (DCP).
We now define the DCP, and the \emph{Strong Discrete Compactness Property}
(SDCP) in our context, (see, e.g., \cite{Kikuchi,SDCP}).
Since our discrete problem~\eqref{eq:kikdisc} uses the bilinear form $\bh$
instead of the $\mathbf{L}^2$-scalar product, everything should be rephrased
accordingly. 
\begin{definition}
{\textbf{Discrete compactness property (DCP).}}
\label{def:DCP}
We say that the \textit{Discrete Compactness Property} holds true for a family
of discrete spaces $(\V_h,Q_h)$, if any sequence $\{\vhn \}_{n=0}^{\infty}$
with $\{\vhn \} \subset \Vhn$, such that
\begin{equation}
\label{eq:DCP_ass}
\aligned
&\|\vhn\|_{div}=1,\\
&\bhn( \vhn, \Curl q_{h_n}) = 0 \quad \forall q_{h_n} \in Q_{h_n},
\endaligned 
\end{equation} 
contains a subsequence (not relabeled) which converges strongly to some 
${\bfv_0}$ in $\Ltwovecomega$, that is
\[
\|\vhn-\bfv_0\|_{0}\rightarrow 0, \quad n\rightarrow\infty.
\]
Here $\{ h_n \}_{n=0}^{\infty}$ is an arbitrary subsequence of our mesh
sequence with $h_n\rightarrow 0$ as $n \rightarrow \infty$. 
 
\end{definition}

%==============================================================
\begin{definition}
{\textbf{Strong discrete compactness property (SDCP).}}
We say that the discrete spaces $(\Vh,Q_h)$ satisfy the 
\textit{Strong Discrete Compactness Property} if they meet the \textit{DCP} 
property with
$$
\Rot \bfv_0 = 0.
$$
\end{definition}

%==============================================================
\begin{theorem}
Let us assume that the seminorm $|\cdot|_h$ is equivalent to the $L^2(\Omega)$
norm.  Then the SDCP holds for $(\Vh,\Qh)$.
\label{th:SDCP}
\end{theorem}

\begin{proof}
Let us consider a sequence $\{\vhn\}\subset\Vhn$ satisfying~\eqref{eq:DCP_ass}. 
In particular, $\{\vhn\}$ belongs to $\Kbhnrot$.
Then Lemma~\ref{lem:lem7_rodo} states that there exist 
$\bfpsi(n)\in\K^{div}$ and $p(n)\in H^{1+s}(\Omega)$, $1/2<s\le 1$ such that
\[
\vhn=\bfpsi(n)+\Grad p(n)
\]
with the following bounds
\[
\aligned
&\|\Grad p(n)\|_{s}\le C\|\Div \vhn\|_{0}\\
&\|\bfpsi(n)\|_{0}\le Ch_n^s\|\Div \vhn\|_{0}.
\endaligned
\]
We have that the sequence $\bfz(n)=\Grad p(n)$ is uniformly bounded in 
$\V\cap\Hroto$.
Thanks to Lemma~\ref{lem:Hodiv_compactness_L2},
$\V\cap\Hroto$ is a compact subspace of $\Ltwovecomega$, therefore there
exists a subsequence of $\bfz(n)$ (denoted by the same index $n$) that
converges strongly to some $\bfz_0$ in $\Ltwovecomega$, that is
\[
\|\bfz(n) - \bfz_0 \|_{0} \rightarrow 0,\quad n\rightarrow \infty. 
\]
From the fact that $\bfz(n)$ is rot-free, we can pass to the limit and obtain
that also $\Rot\bfz_0=0$.
For the same subsequence, we have
\[
\|\vhn-\bfz_0\|_{0} = 
\|\bfpsi(n) + \Grad p(n) -\bfz_0\|_{0}\le
\|\bfpsi(n)\|_{0}+\|\bfz(n)-\bfz_0\|_{0}.
\]

Using the above $L^2$ bounds for $\bfpsi(n)$ and the strong convergence of
$\bfz(n)$ to $\bfz_0$ in $L^2$, we obtain the strong convergence of $\vhn$ to
$\bfz_0$.
\end{proof}
%
%==============================================================
\subsection{SDCP implies EDK, WA and SA}

%==============================================================
%==============================================================
In the following proposition we prove, in the VEM setting, the analogous result
as~\cite[Proposition~3]{SDCP} for the Maxwell's eigenvalue problem (see, also,
\cite{MonkDemko}).  
%==============================================================
\begin{proposition}
\label{prop:SDCP_gives_EDK}
If the SDCP is satisfied for the discrete spaces $(\Vh,\Qh)$, then the EDK
holds true. 
\end{proposition}
% %
\begin{proof}
The proof is by contradiction. 
Let us assume there exists a sequence $\{\uhn\}$ (where $h_n$ 
represents a decreasing sequence of mesh size tending to 
zero as $n\rightarrow \infty$) such that $\uhn \in \Kbhnrot$ that is
\[
 \bhn(\uhn,\Curl q_{h_n}) = 0 \quad \forall q_{h_n}\in \Qhn, 
\]
with 
\[
\|\uhn\|_{0} =1 \quad \text{ and } 
\quad
\|\Div \uhn\|_{0}=\frac{1}{n}, 
\quad
\text{ for } 1 \leq n < \infty.
\]
Since the DCP is satisfied, there exists a subsequence, also denoted 
by $\{\uhn\}$ such that $\uhn\rightarrow\bfu$ strongly in $\Ltwovecomega$,
this implies that $\|\bfu\|_{0}=1$. 
Moreover,  since $\{\uhn\}$ is uniformly bounded in $\V$ then 
$\uhn \rightharpoonup \bfu $ weakly in $\Hdiv$. However, for any $\bfv\in \V$
\begin{equation}
\begin{aligned}
 (\Div \bfu,\Div \bfv) 
   &= \lim_{n\rightarrow\infty}(\Div \uhn,\Div \bfv) \\
%   &\leq  \lim_{n\rightarrow\infty} |(\Div \uhn,\Div \bfv)| \\
   & \leq \lim_{n\rightarrow\infty} \|\Div \uhn\|_{0} \|\Div \bfv\|_{0}  \quad
	\text{ by Cauchy-Schwarz }\\
   & \leq \lim_{n\rightarrow\infty} \frac{1}{n} \|\Div \bfv\|_{0} = 0.
\end{aligned}
\end{equation} 
By choosing $\bfv=\bfu$, we finally get $\|\Div \bfu\|^2_{0} = 0$ and hence
$\Div \bfu = 0$ in $\Omega$.\\
Now, since $\uhn\rightarrow \bfu$ strongly in $\Ltwovecomega$ this means
$\|\uhn - \bfu \|_{0}\rightarrow 0$. 
Also, we have that $\Div \uhn = \frac{1}{n} \rightarrow 0$ with $\Div\bfu=0$,
thus, 
\[
\|\Div (\uhn - \bfu)\|_{0} = \|\Div\uhn\|_{0}\rightarrow0.
\]
Therefore, we conclude that $\uhn \rightarrow \bfu$ strongly in $\Hdiv$ meaning
that 
$$\|\uhn - \bfu\|_{\Div} \rightarrow 0 .$$
Which finally implies that $\bfu\in\V$ because all $\uhn$ have boundary
conditions which are preserved when passing to the limit in the strong
convergence. \\
Moreover, by SDCP the limit $\bfu$ also has the property that $\Rot \bfu = 0$
in $\Omega$ with $\bfu\in\Kbhrot$, and by the Friedrich's inequality we have
\[
\|\bfu\|_{0}\leq C\|\Div \bfu\|_{0}
\]
 where the constant $C$, depends only on the domain. This implies $\bfu=0$ and
hence contradicts $\bfu$ having $\|\bfu\|_{0}=1$.
\end{proof}
%==============================================================

\begin{proposition}
\label{prop:SDCP_gives_WA}
If the SDCP holds true, then the WA of $Q_0$ is satisfied.  
\end{proposition}
\begin{proof}
We shall prove this by contradiction. Let us assume that WA is not valid.
This means that $\exists\epsilon_0>0$ such that we can construct a
decreasing sequence of mesh sizes $\{h_n\}$ tending to zero as $n$ goes to
$\infty$ and a sequence of functions $\{ p(n)\}\subset Q_0$ with the
property that $\|p(n)\|_{Q_0}=1$
such that for all $n$ there exists $\vhn\in \Kbhnrot$ with
$\|\vhn\|_{\Div}=1$ and
\begin{equation}
\label{eq:contradiction}
  (\vhn, \Curl p(n)) \geq \epsilon_0.
\end{equation}
Applying the SDCP to the sequence $\{\vhn\}\in \Kbhnrot$ we can extract a
subsequence (denoted the same) that converges strongly in $\Ltwovecomega$
to some rotation free $\bfv_0$, that is   
\[
\aligned
&  \|\vhn -\bfv_0\|_{0}\rightarrow 0 \text{ for }n\to\infty\\
& \Rot\bfv_0 = 0.
\endaligned
\]
Moreover, there exists $p\in Q$ such up to a subsequence $p(n)$ converges
weakly to $p$. Therefore taking the limit as $n\to\infty$ yields
\[
\lim_{n\to\infty}(\vhn,\Curl p(n))=(\bfv_0,\Curl p)=(\Rot\bfv_0,p)=0,
\]
which contradicts~\eqref{eq:contradiction}.
\end{proof}
%
%==============================================================
%Before going to the the next proposition, we need the following (mild) approximation property on the space sequence $\{\V_h\}$: for any $\bfv_0\in\V_0$ there exists a sequence $\bfv_h\in \V_h$ such that 
%\begin{equation}
%  \label{eq:approx_property}
%  \|\bfv_0-\bfv_h\|_{\Div}\rightarrow 0, \quad \text{ as }h\rightarrow 0,  
%  \end{equation} 
%==============================================================
\begin{proposition}
\label{prop:SDCPAndapproxProp_gives_SA}
  Assume that SDCP holds, 
  the seminorm $|\cdot|_h$ is equivalent to the $L^2$ norm, and the
approximation in Lemma~\ref{lem:interpol_est} satisfied then the SA of $V_0$ holds true.
\end{proposition}
\begin{proof}
By contradiction, assume that the SA is not satisfied. 
This means that $\exists\epsilon_0>0$ such that we can construct a
decreasing sequence of mesh sizes $\{h_n\}$ tending to zero as $n$ goes to
$\infty$ and a sequence of functions $\{ \bfu(n)\}\subset\V_0$ with the
property that $\|\bfu(n)\|_{\V_0}=1$ and 
\begin{equation}  
\label{eq:contradiction_SA}
\inf_{\vhn\in\Kbhnrot}  \|\bfu(n) - \vhn \|_{\Div} \geq \epsilon_0.
\end{equation} 

Since $\V_0$ is compact in $\V$, see Definition~\ref{def:solution_spaces},
there exists a subsequence $\{\bfu(n)\}$ (denoted with the same symbol) that
converges strongly to $\bfu$ in $\V$. 
Moreover, since for each element of the subsequence $\bfu(n)\in\V_0$ it holds
true $\Rot\bfu(n)=0$, then $\Rot \bfu = 0$. As a consequence we have that
$\bfu\in\Grad(H^1(\Omega))$ and thus that $\bfu\in\mathbf{H}^s(\Omega)$ for
some $s>1/2$ thanks to Lemma~\ref{lem:lem2_rodolfo}.
Hence we can define the interpolant $\uIn\in\Vhn$ and, from
Lemma~\ref{lem:interpol_est} we have
\begin{equation}
\label{eq:approx_property}
\|\bfu-\uIn\|_0\le C(h^s\|\bfu\|_{s}+h\|\Div\bfu\|_0).
\end{equation}
Moreover, since $\bfu$ is in the closure of
$\V_0$ and the fact $\Div\uIn=P_k(\Div\bfu)$, we deduce that
\[
    \|\bfu - \uIn \|_{\Div}\rightarrow 0, \text{ as } n\rightarrow \infty.
\]
We observe that in general $\uIn$ does not belong to $\Kbhnrot$.

Let us consider $\phn\in\Qhn$, the solution of the following equation
\begin{equation}
\label{eq:auxiliary_prob}
\bhn(\Curl \phn, \Curl \qhn) = \bhn(\uIn, \Curl \qhn) \quad
\forall\qhn\in\Qhn,
\end{equation} 
which exists thanks to the equivalence of the seminorm $|\cdot|_h$ and the
$L^2$-norm.
Moreover, by choosing the test function $\qhn$ in~\eqref{eq:auxiliary_prob} to
be $\phn$ and utilizing the equivalence of norms, we have
\begin{equation*}
\aligned
  \|\Curl \phn \|^2_{0}&\le C|\Curl \phn |^2_h
  = C\bhn(\Curl \phn, \Curl \phn) &\\ 
  &= C\bhn(\uIn, \Curl \phn)
  \le C|\uIn|_h |\Curl \phn|_h \\
  &\le C\|\uIn\|_{0} \|\Curl \phn\|_{0}.
\endaligned
\end{equation*}
Thus, $\|\Curl \phn \|_{0} \leq C \|\uIn\|_{0}$ and we conclude that
$\|\Curl\phn\|_{\Div}$ is uniformly bounded thanks
to~\eqref{eq:approx_property}.

For each $n>0$, we now consider the element in $\Vhn$ given by
\[
\uhn = \uIn-\Curl \phn,
\]
which belongs to $\Kbhnrot$ due to 
\[
\bhn(\uhn,\Curl\qhn)=\bhn(\uIn-\Curl \phn,\Curl\qhn)=0\quad\forall\qhn\in\Qhn.
\]
The sequence $\{ \uhn\}$ is uniformly bounded in $\V$, indeed we have
\[
\|\uhn\|_{\Div}\leq \|\uIn\|_{\Div} + \|\Curl \phn\|_{\Div}.
\]  
The first norm on the right hand side is bounded since $\uIn$ converges
strongly in $\V$ and we already know that $\Curl\phn$ is uniformly bounded in
$\V$.
From the above we have $\uhn\in\Kbhnrot$ such that the following are satisfied:
\begin{enumerate}
  \item $ \bhn(\uhn, \Curl \qhn) = 0$, for all $\qhn\in\Qhn $
  \item the sequence $\{ \uhn\}$ is uniformly bounded in $\V$ i.e.
    $ \|\uhn\|_{\Div} \leq C$
  \item $\Div \uhn = \Div \uIn$.
\end{enumerate}
Points (1) and (2) meet the conditions of SDCP for $\{\uhn\}$, that guarantees
that there exists a subsequence, still denoted $\{\uhn\}$, which converges
strongly to $\ubar$ in $\Ltwovecomega$ with $\Rot\ubar = 0$.
Moreover, from point (3) we have
\[
\Div \uhn=\Div\uIn\rightarrow \Div \bfu\quad\text{ in }\Ltwovecomega.
\]
In order to contradict~\eqref{eq:contradiction_SA}, we use the triangle
inequality and we get
\[
\|\bfu(n) -\uhn\|_{\Div}\le\|\bfu(n)-\bfu\|_{\Div}+\|\bfu-\uhn\|_{\Div}.
\]
Since the first term on the right hand side tends to zero, the result will follow if we can 
show that $\bfu=\ubar$.\\%
To this end, let us assume that $\bfw =\bfu-\ubar$.
From the definition of $\uhn$, we have 
\[ 
\Curl \phn = \uIn -\uhn\rightarrow\bfu-\ubar=\bfw,\quad\text{as }n\to\infty.
\]
Therefore, 
\[
\bfw = \lim_{n\rightarrow\infty} (\uhn -\uIn) 
\text{ and } 
\Div \bfw = \lim_{n\rightarrow\infty} \Div(\uhn -\uIn) =0.
\]
Moreover, since we proved $\Rot\bfu = 0 $ and $\Rot\ubar = 0 $,
for $\bfw\in\V$, we have $\Rot \bfw=0$ and $\Div\bfw=0$ thus $\bfw=0$.
\end{proof}
%==============================================================
Propositions~\ref{prop:SDCP_gives_EDK},~\ref{prop:SDCP_gives_WA},
and~\ref{prop:SDCPAndapproxProp_gives_SA} show that having SDCP with the
equivalence of norms results in EDK, WA and SA, which in turn imply the
convergence of $T_h$ to $T$, see Theorem~\ref{th:vem_uniform_conv_sol_operator}.

\subsection{Stabilized VEM spaces}

The theory developed above allows to conclude that the stabilized
formulation provides a correct spectral approximation of our problem. This is
the same result as the one obtained in~\cite{Rodolfo} which is now rigorously
proved thanks the proof of Lemma~\ref{lem:interpol_est}.
\begin{theorem}

Let $\Vh$ be a sequence of VEM spaces
as defined in Section \ref{sec:the_virtual_element_discretization}
and consider the discretized
problem~\eqref{eq:pbdisc}, where the bilinear form $\bh$ is given by the
stabilized form $\bhs$. Then the sequence of discrete solution operators
$\{T_h\}$ converges uniformly to the solution operator $T$ in the spirit of
Theorem~\ref{th:vem_uniform_conv_sol_operator}.

\end{theorem}

\begin{proof}

The proof follows the lines of the previous analysis. The only condition that
needs to be checked is that the seminorm $|\cdot|_h$ is equivalent to the
$\Ltwovecomega$ scalar product. In the case of $\bh=\bhs$ this is a
consequence of the presence of the stabilization term
(see~\cite{Rodolfo,basic_principle}).

\end{proof}

%============================================================================
%===================================================================
\section{Stabilization free elements}
\label{sec:stabfree}

Our numerical results, see Section~\ref{sec:numerical_experiments}, show that
in several cases the stabilization is not necessary. A general proof of this
statement is not immediate, but we are able to discuss in more detail the case
of triangular elements of lowest order.

More precisely, we are going to prove rigorously that for lowest order
triangular elements the results of
Theorem~\ref{th:vem_uniform_conv_sol_operator} are valid although the seminorm
$|\cdot|_{h,0}$ is not equivalent to the $L^2(\Omega)$ norm.

We start by showing in the next proposition that the kernel ${\kerbh}_{,0}$ is
reduced to $\{0\}$.

%===============================================================================
\begin{proposition}
\label{pro:tri_rect_mesh_kernel0}
In the case of a triangular mesh and lowest order degree ($k=0$), the space
${\kerbh}_{,0}$ is reduced to $\{0\}$.

\end{proposition}

\begin{proof}

Let $\bfv_h \in\kerbh$  such that $\bho(\bfv_h,\bfv_h) = 0$. 
Our main objective is to show that $\bfv_h = 0$.
By the definition of $\bho$, we have
\[
0 = \bho(\bfv_h,\bfv_h)=\sum_{E\in\T_h} \int_E |\PE\bfv_h|^2,
\] 
which implies that, 
\[
|\PE\bfv_h| = 0, \ \forall E.
\]
By the definition of the projection $\PE$, we have % $\forall q$
for all $q\in\Po_1(E)$
\begin{equation}
\aligned
0 &= \int_E\PE \bfv_h \cdot \Grad q \ d \bfx 
= \int_E \bfv_h \cdot \Grad q\ d\bfx \quad  \\
&= -\int_E \Div \bfv_h \ q \ d\bfx+ \int_{\partial E}\bfv_h \cdot \bfn \ q \ dS \\
&= -\Div \bfv_h \int_E q \ d\bfx + \int_{\partial E} \bfv_h \cdot \bfn \ q \ dS
\endaligned
\label{eq:vanishing_qs}
\end{equation}
where we have used integration by parts and the fact that $\Div \bfv_h$ is a constant. 
Notice that our degrees of freedom on each edge $e_i$ are $\int_{e_i} \bfv_h \cdot \bfn \ \ dS$. 
Looking at triangular elements, we choose two linear functions for $q$ namely,
$q_1$ and $q_2$. In order to evaluate $\int_E q \ d\bfx$ exactly, we choose
a quadrature rule of order $2$ with  quadrature points located at the midpoints $p_1,p_2,p_3$
of each edge. Setting the values of $q_1$ at point $p_1$ to be $+1$, $-1$ at point $p_2$, and zero at point $p_3$. For $q_2$, we have the value at $p_1$ to be $-1$, at $p_2$ to be zero and $1$ at $p_3$. Therefore, having $\int_E q_j \ d\bfx =0$ for $j=1,2$.     
Using that $\bfv_h\cdot\bfn$ is constant on each edge, Equation~\eqref{eq:vanishing_qs} reduces to
\[
0 =  \int_{\partial E} \bfv_h \cdot \bfn \ q_j \  dS = \sum_{i=1}^{3}  \bfv_h \cdot \bfn_i \int_{e_i} q_j \ dS \quad j=1,2 
\]
and we get the system
\[
\aligned
&q_1:&& \bfv_h\cdot\bfn_1|_{e_1}|e_1|-\bfv_h\cdot\bfn_2|_{e_2}|e_2|&&= 0\\
&q_2:&& -\bfv_h\cdot\bfn_1|_{e_1}|e_1|+\bfv_h\cdot\bfn_3|_{e_3}|e_3|&&= 0,
\endaligned
\]
hence, we get
\[
\bfv_h\cdot \bfn_1|_{e_1}|e_1| =
\bfv_h\cdot \bfn_2|_{e_2}|e_2| = 
\bfv_h\cdot \bfn_3|_{e_3}|e_3|.
\]
Therefore, starting from an element $E$ with an edge on the boundary with homogeneous Dirichlet
boundary conditions, we obtain that $\bfv_h$ is vanishing in $E$. Then we can
propagate with similar arguments to the rest of the domain and finally have
$\bfv_h = 0$ everywhere in $\Omega$.
\end{proof}

%===============================================================================

\begin{remark}

The same proof works for meshes of rectangles and lowest order elements. The
last arguments can be used by taking into account the boundary conditions and
\emph{killing} the degrees of freedom of the neighboring elements. For higher
sided polynomials the result is no longer valid (we don't know what is the
threshold, probably already the pentagon is bad).

For higher order schemes we have not investigated the topic theoretically.
From numerical tests and by counting the number of degrees of freedom and the
number of conditions, it seems plausible that the result is true when the
degree is large enough compared to the number of sides of the elements.

\end{remark}

%===============================================================================
The next proposition is the crucial ingredient that will replace the
equivalence of the $|\cdot|_{h,0}$ seminorm and the $L^2(\Omega)$ norm.

\begin{proposition}
\label{pro:equivTR0}
Let us consider a triangular mesh and lowest order degree $k=0$. 
Then $\bho(\bfv_h,\Curl q_h)=0$ for all $q_h\in\Qh$ if and only if
$(\bfv_h,\Curl q_h)=0$ for all $q_h\in\Qh$.
\end{proposition}
\begin{proof}
The discrete space $\Vh$ coincides with the space of Raviart--Thomas elements of
lowest order $\m{RT}_0$. Notice that while the matrix $\m{A}$ associated with 
the bilinear form $(\Div\cdot,\Div\cdot)$ is the same as the one obtained with
Raviart--Thomas elements, the matrix $\m{B}$ is different. In particular, in the
case of $\m{RT}_0$ the matrix $\m{B}$ is computed using the $L^2(\Omega)$
scalar product of the basis functions, while the virtual element matrix
$\m{B}$ is computed using the projection operator defined in~\eqref{eq:PiE}. 
Moreover, we take as $\Qh$ the subspace of $\Honezero$ of piecewise linear
polynomials. Therefore $\Curl q_h$ is a piecewise constant function for all
$q_h\in\Qh$.

Let us consider an element $\bfv_h\in\Vh$, with $\bfv_h\neq0$, such that $\bho(\bfv_h,\Curl q_h)=0$
for all $q_h\in\Qh$.
We have shown in Proposition~\ref{pro:tri_rect_mesh_kernel0} that the kernel
${\kerbh}_{,0}$ is reduced to $\{0\}$, hence $\bfv_h$ is not an element of the
kernel.

Given $\bfv_h\in\Vh$, by definition of $\bho$ we have 
\[
\bho(\bfv_h,\Curl q_h)
=\sum_{E\in\T_h}\int_E\PE\bfv_h\cdot\PE\Curl q_h\, d\bfx 
\quad\forall q_h\in\Qh.
\]
Since $\Curl q_h$ is a piecewise constant, it is the
gradient of some linear polynomial in $E$, that is 
$\Curl q_h|_E\in\Grad\Po_1(E)$ for each $E\in\T_h$. Hence in each
element $E$ $\PE\Curl q_h=\Curl q_h$. Next using the definition of the
projection operator $\PE$ we have
\[
\int_E\PE\bfv_h\cdot\PE\Curl q_h\, d\bfx
=\int_E\PE\bfv_h\cdot\Curl q_h\, d\bfx
=\int_E \bfv_h\cdot\Curl q_h\, d\bfx.
\]
Summing on the element we arrive at
\[
\aligned
\bho(\bfv_h,\Curl q_h)
&=\sum_{E\in\T_h}\int_E\PE\bfv_h\cdot\PE\Curl q_h\, d\bfx\\
&=\sum_{E\in\T_h}\int_E \bfv_h\cdot\Curl q_h\, d\bfx=(\bfv_h,\Curl q_h).
\endaligned
\] 
Hence we conclude that
\[
\bho(\bfv_h,\Curl q_h)=(\bfv_h,\Curl q_h).
\]
Therefore if $\bfv_h\in\Vh$ is such that 
\[
\bho(\bfv_h,\Curl q_h)=0\quad\forall q_h\in\Qh,
\]
then 
\[
(\bfv_h,\Curl q_h)=0\quad\forall q_h\in\Qh
\]
and viceversa.

\end{proof}

We are now in a position to state and prove the main result of this section.

\begin{theorem}

Let $\Vh$ be a sequence of spaces defined on triangular meshes for $k=0$ and
consider the case of the non-stabilized scheme, that is $\bh=\bho$, then
the discrete sequence $\{T_h\}$ converges in norm to $T$ in $\V$ as $h$ tends
to zero, in the spirit of Theorem~\ref{th:vem_uniform_conv_sol_operator}.

\end{theorem}

\begin{proof}

From the abstract setting of Section~\ref{sec:abstract_theory}, if follows
that the only missing property needed for our result is the equivalence of the
$|\cdot|_h$ seminorm and the $L^2(\Omega)$ norm.

We will see that in our case this property is not satisfied (see
Example~\ref{ex:counterexample} below), but we can modify our proofs accordingly as
follows.

The first occurrence where the equivalence of norms has been used, is for the
proof of Lemma~\ref{lem:lem7_rodo}. In this case, we can take advantage of the
equivalence between the orthogonality with respect to $b_h(\cdot,\cdot)$ and
the $L^2$ scalar product, proved in Proposition~\ref{pro:equivTR0}. It turns
out that in this setting Lemma~\ref{lem:lem7_rodo} is identical to the
corresponding Lemma proved in~\cite[Lemma~7]{Rodolfo}. Hence we can use that
result to get the appropriate estimate for the Helmholtz decomposition and to
prove Theorem~\ref{th:SDCP} that gives the SCDP property.

Then all the results of our abstract setting hold true until the crucial proof
of the SA property given in Proposition~\ref{prop:SDCPAndapproxProp_gives_SA}.

Since the space $\Vh$ coincides with the lowest order Raviart--Thomas space
$\m{RT}_0$, the SA property can be obtained directly by using the following
mixed problem: given $\bfu\in\V_0$, find $(\bfu_h,p_h)\in\m{RT}_0\times P_0$
(where $P_0$ is the space of piecewise constant functions with global zero
mean value) such that
\[
\left\{
\aligned
&(\bfu_h,\bfv)+(\Div\bfv,p_h)=0&&\forall\bfv\in\m{RT}_0\\
&(\Div\bfu_h,q)=(\Div\bfu,q)&&\forall q\in P_0.
\endaligned
\right.
\]
Taking $\bfuSA=\bfu_h$ we have that $\bfuSA$ belongs to ${\Kbhrot}$, as it
can be easily seen by choosing $\bfv\in\Curl\Qh$ in the mixed problem, and
satisfies the uniform convergence required for the SA property to be valid.

\end{proof}

We conclude this section by showing that the equivalence between the
$|\cdot|_{h,0}$ seminorm and the $L^2(\Omega)$ norm is not satisfied by the
lowest order triangular elements.

\begin{example}
\label{ex:counterexample}

We construct a sequence of vectors $\bfv_h\in\Vh$ such that
\[
|\bfv_h|_{h,0}\to0\qquad\text{and}\qquad\|\bfv_h\|_0\ge\alpha>0
\]
as $h$ goes to zero.

Let us consider $\Omega=(0,1)^2$ decomposed into $N\times N$ squares, each
divided into two triangles by its diagonal, see Figure~\ref{fg:lucia1}. The
mesh is divided into two regions: the internal one $\T_I$ consisting of
$2(N-2)^2$ elements, highlighted in gray, and the collection $\T_B$ of the
remaining $8(N-1)$ elements touching the boundary. Analogously, the domain
$\Omega$ is seen as the union of $\Omega_I$ and $\Omega_B$.
\begin{figure}
\begin{center}
\includegraphics[width=6cm]{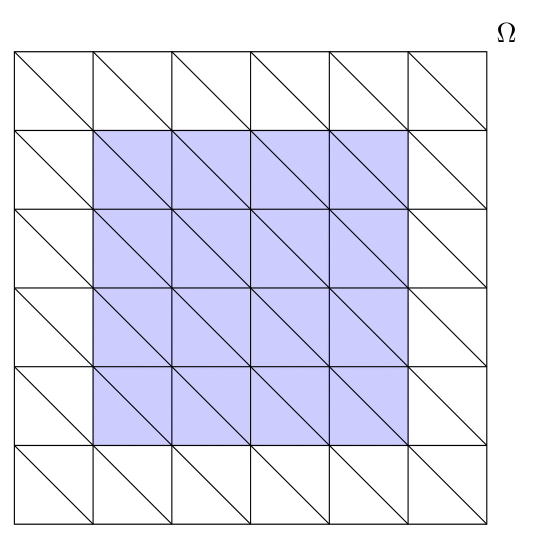}
\caption{Domain with subregions}
\label{fg:lucia1}
\end{center}
\end{figure}

From the proof of Proposition~\ref{pro:tri_rect_mesh_kernel0} it follows
that in each triangle $E\in\T_I$ there exists a function $\bfv_E\in\VE$
such that
\[
\bE(\bfv_E,\bfv_E)=0.
\]
Moreover, these functions can be combined together like in
Figure~\ref{fg:lucia2} in order to construct a function $\bfv_I$ in
$\mathbf{H}(\Div;\Omega_I)$. It turns out that the length of the vectors can
be taken equal to $1$ on the horizontal and vertical sides, and equal to
$1/\sqrt{2}$ along the diagonals.
\begin{figure}
\begin{center}
\includegraphics[width=6cm]{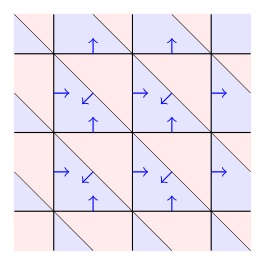}
\caption{Zoom of the interior region $\Omega_I$ with the vectors} 
\label{fg:lucia2}
\end{center}
\end{figure}

In particular, we have the important property that
\[
\sum_{E\in\T_I}\bE(\bfv_E,\bfv_E)=0.
\]

An explicit computation shows that
\[
\int_{\Omega_I}|\bfv_I|^2=\sum_{E\in\T_I}\frac{1}{3N^2}=\frac{2(N-2)^2}{3N^2}\simeq\frac{2}{3}.
\]

The function $\bfv_I$ can be extended to a function $\bfv_h\in\Vh$ defined in
$\Omega$ by adding appropriate pieces on the boundary elements, as depicted in
Figure~\ref{fg:lucia3}.

\begin{figure}
\begin{center}
\includegraphics[width=6cm]{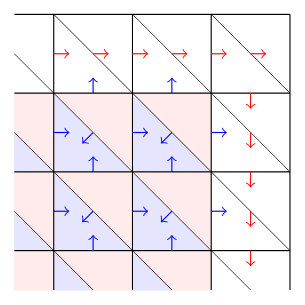}
\caption{Zoom of the right upper corner of $\Omega$ with the vectors}
\label{fg:lucia3}
\end{center}
\end{figure}

Since the area of $\Omega_B$ tends to zero as $N$ goes to infinity, it follows
that 
\[
\bho(\bfv_h,\bfv_h)=\sum_{E\in\T_B}\bE(\bfv_E,\bfv_E)\to0
\]
while
\[
\|\bfv_h\|_0^2\ge\int_{\Omega_I}|\bfv_I|^2\to\frac23
\]
as $N$ tends to infinity.

\end{example}

%============================================================================
\section{Numerical experiments} % (fold)
\label{sec:numerical_experiments}
\begin{figure}[t]
    \centering{}
    \includegraphics[width=0.3\textwidth]{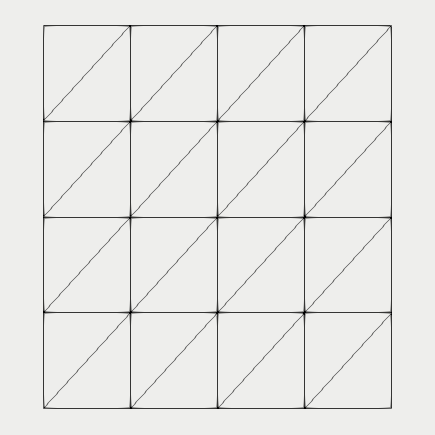}\hfill
    \includegraphics[width=0.3\textwidth]{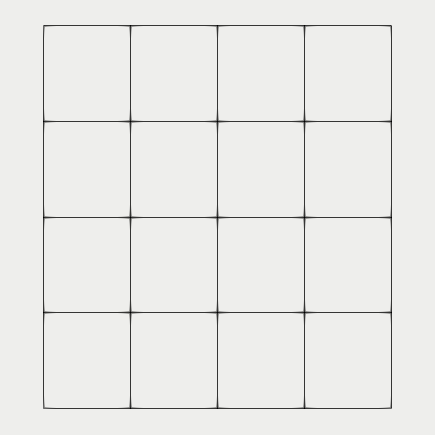}\hfill
    \includegraphics[width=0.3\textwidth]{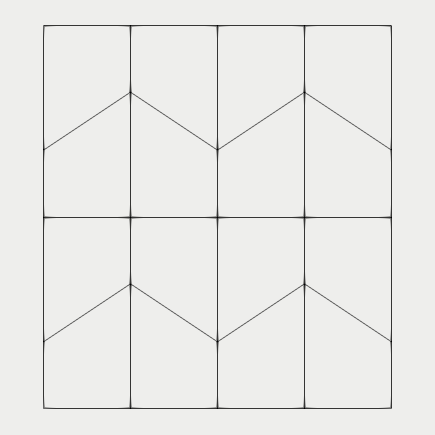}
    \caption{Examples of triangular, square, trapezoidal meshes: 
    $\T_h$ left, $\mathcal{Q}_h$ middle and $\mathcal{Z}_h$ right,
    corresponding to level $\ell=0$ 
}
    \label{fig:meshes_max_4_sided}
\end{figure}
\begin{figure}[t]
\centering{}
    \includegraphics[width=0.3\textwidth]{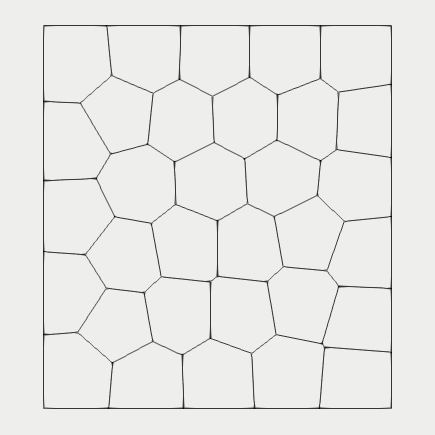}\hspace{0.5cm}
    \includegraphics[width=0.3\textwidth]{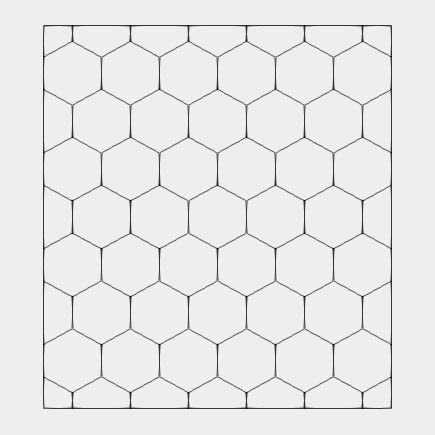}
    \caption{Examples of meshes: $\mathcal{V}_h$(left) and
             $\mathcal{H}_h$(right) corresponding to level $\ell=0$}
    \label{fig:meshes_non_4_sided}   
\end{figure}
In this section we present a series of numerical experiments related to the
scheme that we have analyzed. Specifically, formulation~\eqref{eq:pbdisc0} is
considered. We study two domains, a rectangle and an L-shape domain. In the
case of the rectangular domain we consider several kinds of meshes and we
separate the lowest-order case $k=0$, from the higher order cases with $1\leq k \leq 6$.

The numerical tests introduced in
Section~\ref{subsub:eigenvalues_for_nonstabilized_lowest_order} confirm the
theory presented in Section~\ref{sec:stabfree} where we consider the
non-stabilized lowest order VEM $k=0$ with different mesh structures.
In Section~\ref{subsub:higher_order_nonstabilized_vem} we investigate
non-stabilized higher order cases on the given meshes; it turns out that in
some cases the results are good beyond what we could prove theoretically,
while in several other cases the approximation does not show the desired results.
In Section~\ref{sec:stab} we show that, as already reported in~\cite{Rodolfo},
the stabilization of the mass matrix allows to recover optimal convergence in
perfect agreement with the developed theory.
Finally, Section~\ref{sub:test_2_l_shaped_domain} covers the results for the
non-stabilized and stabilized VEM on the L-shape domain for orders $k=0,1,2$.

We already observed that the divergence of our VEM functions is a polynomial
in each element. Thus, the stiffness matrix $\m{A}$, on the left hand side of
our problem, can be computed exactly and thus does not require
stabilization.
The mass matrix $\m{B}$ on the right hand side can be stabilized or not,
depending on the different tests. In the case where it is not stabilized,
our method does not depend on any parameter.
In general, we obtain the following system:
\[
\m{A}\m{x}=\lambda_h \m{B} \m{x}.
\]

In order to build the above discrete problem, the
\textit{Dune-Vem} module~\cite{dunevem}, which is part of the
\textit{Dune-project}~\cite{dunegridpaperII,dune}, was utilized to define our discrete space
$\Vh$ and construct the matrices of the generalized eigenvalue problem.
This problem is then solved by the Scalable Library for Eigenvalue Problem Computations
(SLEPc)~\cite{slepc}.

%=======================================================================
 \subsection{Test 1: Rectangular domain}
 \label{sec:unit_squ_test}
In this test, we choose the domain $\Omega$ to be the rectangle
$(0,a)\times (0,b)$. Due to the simplicity of this domain, the analytic
eigenvalues are known to be
\[
\lambda_{n,m}:=\pi^2 \left(\left(\frac{n}{a}\right)^2 
+\left(\frac{m}{b}\right)^2  \right) \text{ with } n,m=0,1,2\dots, \ n+m\neq0,
\]
with their associated eigenfunctions
\begin{equation*}
\bfv_{n,m}:=
\begin{pmatrix}
\frac{n}{a}\sin\frac{n\pi x}{a} \cos\frac{m \pi y}{b}\\
\\
\frac{m}{b}\cos\frac{n\pi x}{a} \sin\frac{m \pi y}{b}\\
\end{pmatrix}
\text{ with } n,m=0,1,2\dots, \ n+m\neq0,
\end{equation*}
where we choose $a=1$ and $b=1.1$.

\begin{table}
\caption{
Levels of refinement for hexagonal mesh $\mathcal{H}_h$}
\label{tab:hexg_refinement}
\begin{tabular}{c| c c c c c }
 $\ell$ &  0  &  1  & 2   &3   &4       \\\hline 
$N_\ell$&  59 & 213 & 809 &3153& 12449  \\  
\end{tabular}
\end{table}

In our numerical tests, we consider five different polygonal mesh sequences.
We start from a coarse mesh and refine; hence, we denote the level 
of refinement by $\ell$ and the total number of elements by $N_\ell$. 
For the first four meshes we multiply the total number of elements by $4$ as 
we refine from one level to another. For example, taking level $\ell=1$, we have
\begin{itemize}
\item $\{\T_h \ \}$: structured triangular meshes with $N_1 =128 $
\item $\{\mathcal{Q}_h\}$: uniform rectangular meshes with $N_1 = 64$
\item $\{\mathcal{Z}_h\}$: trapezoidal meshes obtained by
perturbing rectangular meshes with $N_1 = 64$
\item $\{\mathcal{V}_h\}$: Voronoi meshes with $N_1 = 128$
\item $\{\mathcal{H}_h\}$: hexagonal meshes. The sequence in this case does not 
exactly lead to an increase in the number of elements by a factor of four, 
the number of elements on each level is given in Table~\ref{tab:hexg_refinement}. 
\end{itemize}
An example of the adopted meshes for level $\ell=0$ is shown in
Figures~\ref{fig:meshes_max_4_sided} and~\ref{fig:meshes_non_4_sided}. 

%============================================================================

%============================================================================
%
In general, we are seeking for the eigenpair $(\lambda_{h,i},\bfu_{h,i})$, and in
our tables we report the scaled eigenvalues
$\hat\lambda_{h,i}:=\frac{\lambda_{h,i}}{\pi^2}$. In each table the rows
contain the first $7$ eigenvalues, and the columns represent either a
refinement of the mesh in the case of $h$ refinement as we move from one level
to another, or the order $k$ of the VEM space in case of $k$ refinement.
We also depict the rate of convergence calculated for any two consecutive refinements.
\subsubsection{Eigenvalues with non-stabilized mass matrix - the lowest order case $k=0$} % (fold)
\label{subsub:eigenvalues_for_nonstabilized_lowest_order}
In this part, we report results for the non-stabilized lowest order case
with $k=0$. Table~\ref{tab:tri_squ_no_stab_slepc} presents the approximation of
the first 7 eigenvalues for triangular, rectangular and trapezoidal elements. As
proven in Section~\ref{sec:stabfree} the choice of triangular elements shows
good approximation. The convergence is optimal considering that eigenfunctions
of Problem~\eqref{eq:pbdisc0} are smooth. Moreover, approximation on the mesh
of rectangles and trapezoids also shows good convergence
of eigenvalues, although at the moment we do not have a proof available.

Table~\ref{tab:voron_hexg_no_stab_slepc} presents the
results on Voronoi $\{\mathcal{V}_h\}$ and hexagonal $\{\mathcal{H}_h\}$ meshes
where it is clear that optimal convergence is not achieved.

\begin{table}[t]%%% This table was produced with BETA=0 and SLEPC
\caption{
\label{tab:tri_squ_no_stab_slepc}
Relative error for the first 7 eigenvalues on meshes $\mathcal{T}_h$, 
$\mathcal{Q}_h$ and $\mathcal{Z}_h$ with
non-stabilized mass matrix for lowest order $k=0$ }
\footnotesize
\centering
\begin{tabular}{c|c c c c}
\hline
Exact&\multicolumn{4}{c}{Relative error (rate)}\\[3pt]
%%%%%%============================Triangules-BETA=0, k=0 
\hline
\multicolumn{5}{c}{$\T_h$ }\\
\hline 
       0.826446 & 6.72e-04 & 1.81e-04 (1.89) & 4.62e-05 (1.97) & 1.16e-05 (1.99)  \\ 
       1.000000 & 8.39e-04 & 2.21e-04 (1.92) & 5.60e-05 (1.98) & 1.40e-05 (2.00)  \\ 
       1.826446 & 1.27e-02 & 3.17e-03 (2.00) & 7.93e-04 (2.00) & 1.98e-04 (2.00)  \\ 
       3.305785 & 2.81e-03 & 7.35e-04 (1.94) & 1.85e-04 (1.99) & 4.64e-05 (2.00)  \\ 
       4.000000 & 3.69e-03 & 9.04e-04 (2.03) & 2.25e-04 (2.01) & 5.62e-05 (2.00)  \\ 
       4.305785 & 2.24e-02 & 5.88e-03 (1.93) & 1.49e-03 (1.98) & 3.73e-04 (2.00)  \\ 
       4.826446 & 1.82e-02 & 4.25e-03 (2.10) & 1.04e-03 (2.03) & 2.60e-04 (2.01)  \\ 
   %%%%%%============================SQUARES-BETA=0, k=0 
   \hline 
   \multicolumn{5}{c}{ $\mathcal{Q}_h$} \\
   \hline
      0.826446 & 2.63e-02  & 6.46e-03 (2.02) & 1.61e-03 (2.01) & 4.02e-04 (2.00) \\ 
      1.000000 & 2.63e-02  & 6.46e-03 (2.02) & 1.61e-03 (2.01) & 4.02e-04 (2.00) \\ 
      1.826446 & 2.63e-02  & 6.46e-03 (2.02) & 1.61e-03 (2.01) & 4.02e-04 (2.00) \\ 
      3.305785 & 1.13e-01  & 2.63e-02 (2.10) & 6.46e-03 (2.02) & 1.61e-03 (2.01) \\ 
      4.000000 & 1.13e-01  & 2.63e-02 (2.10) & 6.46e-03 (2.02) & 1.61e-03 (2.01) \\ 
      4.305785 & 9.25e-02  & 2.17e-02 (2.09) & 5.33e-03 (2.02) & 1.33e-03 (2.01) \\ 
      4.826446 & 9.78e-02  & 2.29e-02 (2.10) & 5.63e-03 (2.02) & 1.40e-03 (2.01) \\ 
% % %%%%%%============================Disorted quads -BETA=0, k=0  
   \hline 
   \multicolumn{5}{c}{ $\mathcal{Z}_h$} \\
   \hline 
  0.826446 & 2.63e-02  & 6.46e-03 (2.02) & 1.61e-03 (2.01) & 4.02e-04 (2.00)  \\ 
  1.000000 & 2.39e-02  & 5.88e-03 (2.02) & 1.47e-03 (2.01) & 3.66e-04 (2.00)  \\ 
  1.826446 & 2.68e-02  & 6.67e-03 (2.01) & 1.66e-03 (2.00) & 4.16e-04 (2.00)  \\ 
  3.305785 & 1.13e-01  & 2.63e-02 (2.10) & 6.46e-03 (2.02) & 1.61e-03 (2.01)  \\ 
  4.000000 & 1.06e-01  & 2.50e-02 (2.09) & 6.15e-03 (2.02) & 1.53e-03 (2.01)  \\ 
  4.305785 & 9.81e-02  & 2.31e-02 (2.09) & 5.68e-03 (2.02) & 1.41e-03 (2.01)  \\ 
  4.826446 & 9.24e-02  & 2.21e-02 (2.06) & 5.47e-03 (2.02) & 1.36e-03 (2.00)  \\ 
    \hline
    $\ell$ & $1$ & $2$ & $3$ & $4$ \\
    % Mesh & $8$ & $16$ & $32$ & $64$ \\
    \hline
 \end{tabular}
\end{table}
%%%%
\begin{table}[t]
\caption{
\label{tab:voron_hexg_no_stab_slepc}
Relative error for the first 7 eigenvalues on meshes $\mathcal{V}_h$ and $\mathcal{H}_h$ with
non-stabilized mass matrix and lowest order $k=0$}
\footnotesize
\centering
\begin{tabular}{c|c c c c}
\hline
Exact&\multicolumn{4}{c}{Relative error (rate)}\\[3pt]
  \hline 
   \multicolumn{5}{c}{ $\mathcal{V}_h$} \\
   \hline 
  %%%%%%============================VORONOI-BETA=0, k=0 
  0.826446 & 1.59e-02 & 4.39e-03 (1.80) & 1.31e-03 (1.71) & 5.65e-04 (1.21) \\ 
  1.000000 & 2.38e-02 & 6.72e-03 (1.77) & 2.04e-03 (1.69) & 6.52e-04 (1.64) \\ 
  1.826446 & 4.58e-02 & 1.55e-02 (1.52) & 3.88e-03 (1.96) & 2.38e-03 (0.70) \\ 
  3.305785 & 6.98e-02 & 2.44e-02 (1.47) & 9.33e-03 (1.37) & 2.93e-03 (1.66) \\ 
  4.000000 & 1.22e-01 & 3.09e-02 (1.92) & 1.17e-02 (1.38) & 3.95e-03 (1.55) \\ 
  4.305785 & 1.24e-01 & 6.21e-02 (0.97) & 1.49e-02 (2.03) & 5.08e-03 (1.54) \\ 
  4.826446 & 2.08e-01 & 5.00e-02 (1.99) & 1.88e-02 (1.39) & 5.64e-03 (1.72) \\ 
    \hline
   \multicolumn{5}{c}{ $\mathcal{H}_h$} \\
   \hline 
   0.826446 & 6.18e-03 & 1.57e-03 (2.00) & 3.97e-04 (2.00) & 9.99e-05 (2.00) \\ 
   1.000000 & 1.08e-01 & 9.69e-02 (0.15) & 9.29e-02 (0.06) & 9.12e-02 (0.03) \\ 
   1.826446 & 5.07e-02 & 3.54e-02 (0.52) & 3.05e-02 (0.22) & 2.87e-02 (0.09) \\ 
   3.305785 & 2.51e-02 & 6.32e-03 (2.01) & 1.59e-03 (2.00) & 4.00e-04 (2.00) \\ 
   4.000000 & 2.92e-01 & 2.50e-01 (0.23) & 2.39e-01 (0.07) & 2.35e-01 (0.02) \\ 
   4.305785 & 2.78e-01 & 2.40e-01 (0.22) & 2.28e-01 (0.07) & 2.24e-01 (0.03) \\ 
   4.826446 & 6.30e-01 & 5.63e-01 (0.16) & 5.47e-01 (0.04) & 5.42e-01 (0.01) \\ 
    \hline
    $\ell$ & $1$ & $2$ & $3$ & $4$ \\ 
    % Mesh & $12$ & $24$ & $48$ & $96$ \\
    \hline 
\end{tabular}
% }
\end{table}

%============================================================================
%============================================================================
\subsubsection{Eigenvalues with non-stabilized mass matrix - higher order
cases} % (fold)
\label{subsub:higher_order_nonstabilized_vem}
%============================================================================
%============================================================================

%%%%%%%============================Trig BETA=0 Higher order ===============================
\begin{table}[t]
\caption{ \label{tab:tri_squ_no_stab_higher_order}
Relative error for the first 7 eigenvalues on the triangular mesh $\mathcal{T}_h$: 
non-stabilized mass matrix with orders $k = 1,2,3$ on level two ($\ell=2$)
}
\footnotesize
\centering 
\begin{tabular}{ c |c c c }\hline
Exact   & $k=1$ & $k=2$ & $k=3$ \\[3pt]
 \hline
   0.826446 & 1.19e-06 & 4.35e-10 & 9.66e-14  \\ 
   1.000000 & 1.10e-06 & 4.24e-10 & 9.75e-14  \\ 
   1.826446 & 1.13e-05 & 1.06e-08 & 5.52e-12  \\ 
   3.305785 & 2.07e-05 & 2.77e-08 & 1.85e-11  \\ 
   4.000000 & 2.00e-05 & 2.70e-08 & 1.84e-11  \\ 
   4.305785 & 5.79e-05 & 1.19e-07 & 1.40e-10  \\ 
   4.826446 & 5.44e-05 & 1.09e-07 & 1.32e-10  \\ 
 \hline
\end{tabular}
\end{table}
%%
%%%%%%%============================Voronoi-BETA=0 Higher order ===============================
%

\begin{table}[t]
\caption{ \label{tab:voro_no_stab_higher_order}
Relative error for the first 7 eigenvalues on the Voronoi mesh $\mathcal{V}_h$:
non-stabilized mass matrix with orders $k = 1,2,3$ on level two ($\ell=2$)}
\footnotesize
\centering
\begin{tabular}{ c| c c c} \hline
 Exact & $k=1$ & $k=2$ & $k=3$ \\[3pt]
 \hline
  0.826446 & 6.28e-04 & 4.46e-09 & 1.98e-12  \\ 
  1.000000 & 9.90e-04 & 9.16e-09 & 5.75e-12  \\ 
  1.826446 & 3.04e-03 & 1.15e-07 & 7.20e-11  \\ 
  3.305785 & 6.91e-03 & 7.43e-07 & 4.31e-10  \\ 
  4.000000 & 9.59e-03 & 1.57e-06 & 8.00e-10  \\ 
  4.305785 & 1.91e-02 & 2.63e-06 & 2.50e-09  \\ 
  4.826446 & 1.62e-02 & 3.94e-06 & 3.53e-09  \\
 \hline
 \end{tabular}
\end{table}  
% }
%
%%%%%%%============================Hexagon-BETA=0 Higher even orders ===============================
%
\begin{table}[t]
 \caption{\label{tab:hex_no_stab_even_higher_order}
Relative error for the first 7 eigenvalues on the hexagonal mesh $\mathcal{H}_h$: non-stabilized mass
matrix with even orders $k = 2,4,6$ on level two ($\ell=2$)}
\footnotesize
\centering
\begin{tabular}{c|c c c } \hline
 Exact & $k=2$ & $k=4$ & $k=6$ \\[3pt]
 \hline
    0.826446 & 5.33e-11 & 8.70e-13 & 2.70e-13  \\ 
    1.000000 & 1.45e-06 & 1.61e-12 & 1.00e-11 \\ 
    1.826446 & 8.82e-07 & 4.62e-13 & 1.96e-13 \\ 
    3.305785 & 3.35e-09 & 7.68e-14 & 3.95e-14 \\ 
    4.000000 & 3.76e-05 & 1.07e-12 & 1.33e-15 \\ 
    4.305785 & 1.81e-04 & 2.74e-12 & 2.70e-14 \\ 
    4.826446 & 5.14e-06 & 3.55e-13 & 1.29e-14 \\
 \hline
 \end{tabular}
\end{table}

%
%%%%%%%============================Hexagon-BETA=0 Higher odd orders ===============================
%
\begin{table}[t]
\caption{ \label{tab:hex_no_stab_odd_higher_order}
First 7 eigenvalues on the hexagonal mesh $\mathcal{H}_h$: non-stabilized mass matrix with odd orders $k=1,3,5$ on level two ($\ell=2$). The well approximating
eigenvalues are highlighted}
\centering
\footnotesize
  \begin{tabular}{c c cc cc cc}\hline
$\hat\lambda_{h,i}$ & Exact & \multicolumn{2}{c}{$k=1$}  &
\multicolumn{2}{ c }{ $k=3$}  &  \multicolumn{2}{c}{$k=5$} \\
&  &  real &  imaginary &  real & imaginary &  real &  imaginary\\
\hline
  1 & 0.826446  & 0.139411 & -843.980759 & 0.223144 & -529.593449 & 0.160601 & -1144.73996  \\
  2 & 1.000000  & 0.139411 &  843.980759 & 0.223144 &  529.593449 & 0.160601 &  1144.73996  \\
  3 & 1.826446  & 0.280881 & -679.584440 & 0.395428 & -403.726161 & 0.195456 & -381.803706  \\
    &           & \vdots   &  \vdots     &  \vdots  &  \vdots     & \vdots   &  \vdots       \\
  7  & 4.826446 & 0.556690 & -537.308077 & 0.624345 & -696.587871 &
{\color{teal}\bf 0.826446}  & {\color{teal}\bf 0 } \\
  8  & 7.305785 & 0.556690 &  537.308077 & 0.624345 &  696.587871 &
{\color{teal}\bf1.000000}  & {\color{teal}\bf 0 } \\
  9  & 7.438017 & {\color{teal}\bf0.826447}  & {\color{teal} \bf0 }       & 0.820943 &-559.430831  & 1.298910 & -863.907733 \\
  10 & 8.438017 & 0.855616 & -711.505977  & 0.820943 & 559.430831  & 1.298910 & 863.907733 \\
  11 & 9.000000 & 0.855616 &  711.505977  & {\color{teal}\bf 0.826446} &
{\color{teal}\bf 0 }      & 1.587927 & -1026.2485 \\ \hline
\end{tabular}
\end{table}
Within the \textit{Dune-project} it is pretty straightforward to use higher order
VEM spaces. To our knowledge there is no other package that has this
property for VEM. Although no theory has been developed in the case of non
stabilized right hand side, we investigate in this
section higher order VEM approximation spaces with a fixed level ($\ell=2)$
for each case. As in the lowest order case, when taking higher
orders for triangular, rectangular and trapezoidal elements with orders being
$k=1,2,3$, good approximation of eigenvalues is achieved and one can
appreciate the better accuracy of the eigenvalues as the degree increases. As
there is no
significant difference in the results between different types of meshes,
Table~\ref{tab:tri_squ_no_stab_higher_order} only reports the results for
triangular mesh $\mathcal{T}_h$.

The use of VEM spaces of order $k=1,2,3$ on Voronoi meshes 
$\mathcal{V}_h$ produces good results. Indeed,
Table~\ref{tab:voro_no_stab_higher_order} shows that the first $7$ eigenvalues
are well approximated unlike the lowest order case $k=0$ reported in
Table~\ref{tab:voron_hexg_no_stab_slepc}.
On the contrary, a different finding can be observed for higher orders on
hexagonal meshes $\mathcal{H}_h$.
A peculiar phenomena has been found for even and odd orders. The first $7$
eigenvalues for orders $k=2,4,6$ are well approximated as shown in
Table~\ref{tab:hex_no_stab_even_higher_order}.
On the other hand, for higher order odd values, $k=1,3,5$, several
spurious complex eigenvalues appear. This is detailed in
Table~\ref{tab:hex_no_stab_odd_higher_order}, where \textit{good} eigenvalues
are those with vanishing imaginary part. The complex eigenvalues are clearly
associated with singular pencils as described in Section~\ref{sec:pencil}.
These VEM spaces obviously are not useful in the numerical approximation of
the problem.

In order to confirm the good behavior of the even higher order hexagonal
case, we plot also the eigenfunction corresponding to the first, third, fifth and sixth
eigenvalue for $k=2$ in Figure~\ref{fig:eigenfunction_order_2_hexg}.
\begin{figure}[t]
\includegraphics[width=.24\textwidth]{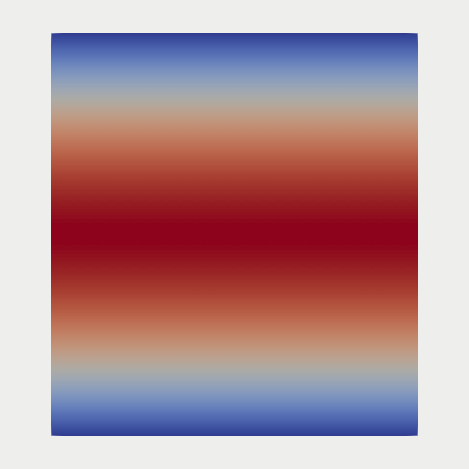}
\includegraphics[width=.24\textwidth]{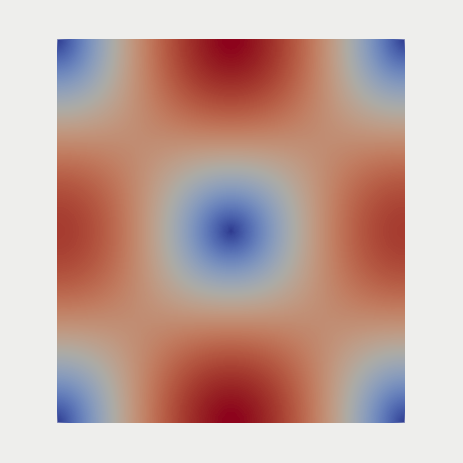}
\includegraphics[width=.24\textwidth]{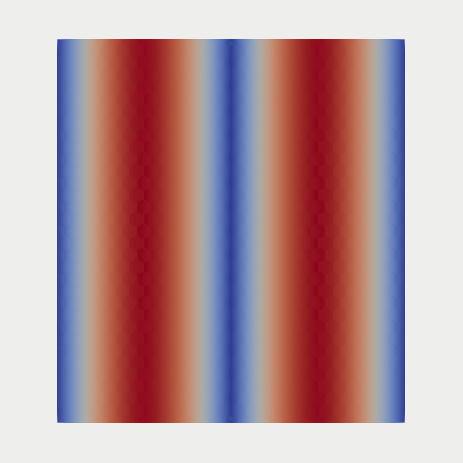}
\includegraphics[width=.24\textwidth]{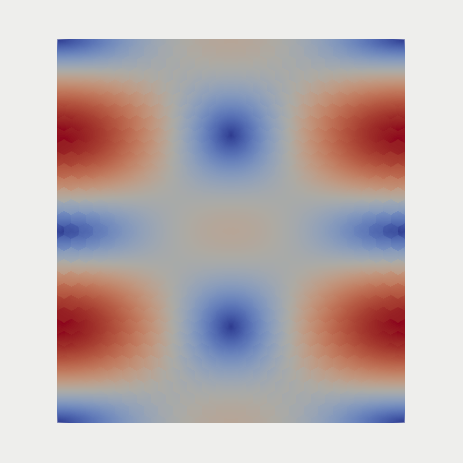}

\caption{Modulus of the first, third, fifth and sixth eigenfunctions on the hexagonal mesh ($\ell=2$) with
order $k=2$}
\label{fig:eigenfunction_order_2_hexg}
\end{figure}
% subsection higher_order_nonstabilized_vem (end)

%=======================================================================
%=======================================================================
\subsubsection{Eigenvalues with stabilized mass matrix}
\label{sec:stab}
%=======================================================================
%=======================================================================

\begin{table}[t]
\caption{
 \label{tab:voron_hex_stab_lowest_degree}
Relative error for the first $7$ eigenvalues on meshes $\mathcal{V}_h$ and $\mathcal{H}_h$ with
stabilized mass matrix $(\sigma_E=0.1)$ and $k=0$}
\footnotesize
\centering
\begin{tabular}{c|c c c c}
\hline
Exact&\multicolumn{4}{c}{Relative error (rate)}\\[3pt]
\hline
\multicolumn{5}{c}{$\mathcal{V}_h$}\\
\hline
  0.826446 & 7.33e-03  & 1.80e-03 (1.96) & 4.46e-04 (1.98) & 1.11e-04 (1.99) \\ 
  1.000000 & 9.16e-03  & 2.23e-03 (1.98) & 5.46e-04 (2.00) & 1.35e-04 (2.01) \\ 
  1.826446 & 1.61e-02  & 3.96e-03 (1.96) & 9.90e-04 (1.97) & 2.45e-04 (2.00) \\ 
  3.305785 & 3.00e-02  & 7.29e-03 (1.98) & 1.78e-03 (2.00) & 4.45e-04 (1.99) \\ 
  4.000000 & 3.71e-02  & 8.96e-03 (1.99) & 2.18e-03 (2.01) & 5.39e-04 (2.00) \\ 
  4.305785 & 3.87e-02  & 9.50e-03 (1.96) & 2.31e-03 (2.01) & 5.77e-04 (1.99) \\ 
  4.826446 & 4.40e-02  & 1.07e-02 (1.98) & 2.62e-03 (1.99) & 6.50e-04 (2.00) \\ 
\hline
\multicolumn{5}{c}{$\mathcal{H}_h$}\\
\hline 
%%%%%%============================Hexg-BETA=0.1, k=0 
  0.826446 & 4.51e-03  & 1.12e-03 (2.03) & 2.82e-04 (2.01) & 7.03e-05 (2.01) \\ 
  1.000000 & 6.12e-03  & 1.52e-03 (2.03) & 3.80e-04 (2.01) & 9.50e-05 (2.01) \\ 
  1.826446 & 1.09e-02  & 2.68e-03 (2.04) & 6.68e-04 (2.02) & 1.67e-04 (2.01) \\ 
  3.305785 & 1.81e-02  & 4.51e-03 (2.03) & 1.12e-03 (2.02) & 2.81e-04 (2.00) \\ 
  4.000000 & 2.49e-02  & 6.11e-03 (2.05) & 1.52e-03 (2.02) & 3.80e-04 (2.01) \\ 
  4.305785 & 2.49e-02  & 6.12e-03 (2.05) & 1.52e-03 (2.02) & 3.79e-04 (2.01) \\ 
  4.826446 & 2.99e-02  & 7.30e-03 (2.06) & 1.82e-03 (2.02) & 4.53e-04 (2.01) \\ 
  \hline
$\ell$ & $1$ & $2$ & $3$ & $4$ \\
    % Mesh & $12$ & $24$ & $48$ & $96$ \\
\hline
 \end{tabular}
\end{table}
%%%%%%%%%%%%

\begin{table}[t]
\caption{ \label{tab:hex_stab_odd_higher_order}
Relative errors for the first $7$ eigenvalues on the mesh $\mathcal{H}_h$: stabilized mass
matrix $(\sigma_E=0.1)$ and orders $k=1,3,5$ on level two ($\ell=2$)}
\footnotesize
\centering
\begin{tabular}{c|c c c } \hline
 Exact & $k=1$ & $k=3$ & $k=5$ \\[3pt]
 \hline
   0.826446 & 2.19e-07 & 4.16e-14 & 1.11e-14  \\ 
   1.000000 & 3.61e-07 & 4.05e-13 & 1.98e-13  \\ 
   1.826446 & 1.23e-06 & 9.41e-14 & 1.63e-14  \\ 
   3.305785 & 3.49e-06 & 1.43e-13 & 5.02e-14  \\ 
   4.000000 & 5.78e-06 & 1.67e-13 & 1.09e-14  \\ 
   4.305785 & 5.83e-06 & 1.82e-13 & 1.05e-14  \\ 
   4.826446 & 9.23e-06 & 1.36e-12 & 8.47e-15  \\ 
\hline
 \end{tabular}
\end{table}

We conclude this subsection on the numerical results for the rectangular domain,
by showing that in the case of lowest order and stabilized mass matrix the
eigenvalues are computed correctly even in cases where the non stabilized
matrix gave unreliable results. 
We consider the same natural stabilization as the one presented in
~\cite{Rodolfo}, namely 
\[
\SE(\bfu_h,\bfv_h)=\sigma_E\sum_{k=1}^{N_E}\left(\int_{e_k}\bfu_h\cdot\bfn\right)
\left(\int_{e_k}\bfv_h\cdot\bfn\right)\qquad \bfu_h,\ \bfv_h\in\VE,
\]
where $N_E$ is the number of edges of $E$.
Table~\ref{tab:voron_hex_stab_lowest_degree} reports on the results obtained
in the lowest order case with both the Voronoi and hexagonal meshes  while
Table~\ref{tab:hex_stab_odd_higher_order} presents results for higher order
odd values for hexagonal mesh on $\ell=2$.

%============================================================================
%============================================================================
\subsection{Test 2: L-shaped domain} % (fold)
\label{sub:test_2_l_shaped_domain}

In this test we consider the L-shaped domain, as in the previous test, 
where now we remove the lower right square of the domain. Hence, it consists of the union of three squares subdivided into $2^{l+1}$ subsquares.
The domain chosen is such that 
$\Omega=(-1,1)^2\setminus[(0,1)\times(-1,0)]$ for which we have the reference
solutions provided in~\cite{l_shape_benchmark}. 

We consider a sequence of triangular and square meshes: a sample of the two
sequences is reported in Figure~\ref{fig:L_shapes}.
\begin{figure}[t]
\centering
\includegraphics[width=.3\textwidth]{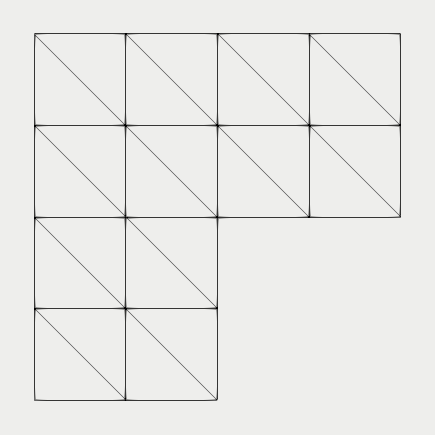}
\hspace{0.5cm}
\includegraphics[width=.3\textwidth]{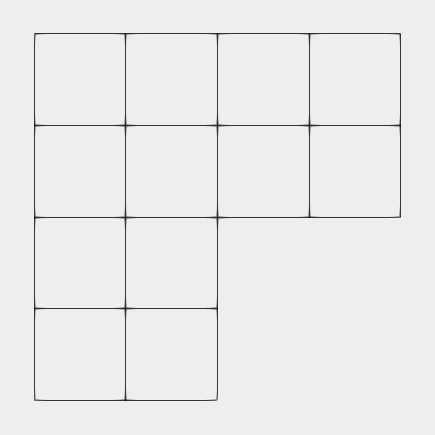}
\caption{Examples of triangular and square structured meshes: $\mathcal{LT}_h$
left and $\mathcal{LQ}_h$ right, corresponding to level $\ell=0$}
\label{fig:L_shapes}
\end{figure}

%%%%%%%%%%%%%%%%%%%%%%%%%%%
\subsubsection{Eigenvalues with non-stabilized mass matrix}
%%%%%%%%%%%%%%%%%%%%%%%%%%
Table~\ref{tb:ns_L_tri} shows the results corresponding to the mesh sequence
$\mathcal{LT}_h$ and order equal to $k=0,1,2$. It can be seen that the results
are in perfect agreement with what one would expect: the first and fifth
eigenvalues correspond to singular eigenfunctions, the second one has $H^1$
regular eigenfunction, and the remaining third and fourth eigenvalues
correspond to analytic eigenfunctions.

\begin{table}[t]
 \caption{Triangular mesh on the L-shaped domain: computations with
non-stabilized mass matrix}
\label{tb:ns_L_tri}
\footnotesize
    \centering
\begin{tabular}{c|c c c c c c c c c c}
\hline
 Exact&\multicolumn{4}{c}{Relative error (rate)}\\[3pt]
\hline
\multicolumn{5}{c}{$k=0$}\\
\hline
  1.475622  & 1.53e-02  & 5.98e-03 (1.36) & 2.35e-03 (1.35) & 9.24e-04 (1.34)  \\
  3.534031  & 1.83e-03  & 4.35e-04 (2.07) & 1.05e-04 (2.05) & 2.58e-05 (2.03)  \\ 
  9.869604  & 1.17e-03  & 2.90e-04 (2.01) & 7.24e-05 (2.00) & 1.81e-05 (2.00)  \\ 
  9.869604  & 1.13e-03  & 2.88e-04 (1.97) & 7.22e-05 (1.99) & 1.81e-05 (2.00)  \\ 
  11.389479 & 6.18e-03  & 1.53e-03 (2.01) & 3.82e-04 (2.01) & 9.52e-05 (2.00)  \\
 \hline 
\multicolumn{5}{c}{$k=1$}\\
\hline
  1.475622  & 8.60e-04  & 3.37e-04 (1.35) & 1.33e-04 (1.34) & 5.27e-05 (1.34) \\ 
  3.534031  & 1.23e-05  & 1.38e-06 (3.15) & 1.82e-07 (2.93) & 2.63e-08 (2.79) \\ 
  9.869604  & 1.28e-05  & 8.12e-07 (3.98) & 5.09e-08 (4.00) & 3.18e-09 (4.00) \\ 
  9.869604  & 2.49e-05  & 1.57e-06 (3.99) & 9.80e-08 (4.00) & 6.12e-09 (4.00) \\ 
  11.389479 & 5.73e-05  & 3.83e-06 (3.90) & 2.80e-07 (3.78) & 2.40e-08 (3.54) \\ 
\hline 
\multicolumn{5}{c}{$k=2$}\\
\hline
 1.475622   & 1.18e-04  & 4.68e-05 (1.33) & 1.86e-05 (1.33) & 7.38e-06 (1.33) \\ 
 3.534031   & 2.19e-06  & 3.42e-07 (2.67) & 5.39e-08 (2.67) & 8.48e-09 (2.67) \\ 
 9.869604   & 2.59e-08  & 4.09e-10 (5.99) & 6.18e-12 (6.05) & 2.14e-13 (4.85) \\ 
 9.869604   & 2.87e-08  & 4.53e-10 (5.99) & 7.08e-12 (6.00) & 6.92e-13 (3.35) \\ 
 11.389479  & 1.08e-06  & 1.54e-07 (2.81) & 2.40e-08 (2.68) & 3.78e-09 (2.67) \\ 
 \hline
    $\ell$ & $1$ & $2$ & $3$ & $4$ \\
\hline
  \end{tabular}
\end{table}

The results on the square mesh sequence are reported in
Table~\ref{tb:ns_L_squ} and present a surprising behavior in the lowest order
case when $k=0$. In particular, the first eigenvalue, corresponding to the
eigenfunction with a strong singularity, is approximated with second order
accuracy.
We observed this behavior with the dune code and confirmed these findings using
a code written in MATLAB. Results from both codes showed the same
superconvergence phenomenon.

\begin{table}[t]
 \caption{Square mesh on the L-shaped domain: computations with 
non-stabilized mass matrix}
\label{tb:ns_L_squ}
\footnotesize
    \centering
\begin{tabular}{c|c c c c c c c c c c}
\hline
 Exact&\multicolumn{4}{c}{Relative error (rate)}\\[3pt]
\hline
\multicolumn{5}{c}{$k=0$}\\
\hline
   1.475622  & 4.57e-03  & 1.16e-03 (1.98) & 2.91e-04 (1.99) & 7.32e-05 (1.99)   \\ 
   3.534031  & 7.37e-03  & 1.84e-03 (2.01) & 4.58e-04 (2.00) & 1.15e-04 (2.00)  \\ 
   9.869604  & 2.63e-02  & 6.46e-03 (2.02) & 1.61e-03 (2.01) & 4.02e-04 (2.00)  \\ 
   9.869604  & 2.63e-02  & 6.46e-03 (2.02) & 1.61e-03 (2.01) & 4.02e-04 (2.00)  \\ 
   11.389479 & 2.32e-02  & 5.71e-03 (2.02) & 1.42e-03 (2.01) & 3.55e-04 (2.00) \\ 
 \hline 
\multicolumn{5}{c}{$k=1$}\\
\hline
    1.475622  & 5.76e-03  & 2.47e-03 (1.22) & 1.03e-03 (1.27) & 4.19e-04 (1.29) \\ 
    3.534031  & 2.25e-04  & 5.93e-05 (1.92) & 1.55e-05 (1.94) & 3.99e-06 (1.96) \\ 
    9.869604  & 6.55e-05  & 4.12e-06 (3.99) & 2.58e-07 (4.00) & 1.61e-08 (4.00) \\ 
    9.869604  & 6.55e-05  & 4.12e-06 (3.99) & 2.58e-07 (4.00) & 1.61e-08 (4.00) \\ 
    11.389479 & 2.28e-03  & 5.58e-04 (2.03) & 1.39e-04 (2.01) & 3.47e-05 (2.00) \\ 

\hline 
\multicolumn{5}{c}{$k=2$}\\
\hline
  1.475622  & 1.94e-03  & 7.72e-04 (1.33) & 3.07e-04 (1.33) & 1.22e-04 (1.33) \\ 
  3.534031  & 1.37e-05  & 2.16e-06 (2.66) & 3.41e-07 (2.67) & 5.37e-08 (2.67) \\ 
  9.869604  & 7.23e-08  & 1.14e-09 (5.99) & 1.78e-11 (6.00) & 1.98e-13 (6.49) \\ 
  9.869604  & 7.23e-08  & 1.14e-09 (5.99) & 1.78e-11 (6.00) & 2.32e-13 (6.26) \\ 
  11.389479 & 8.56e-06  & 1.11e-06 (2.95) & 1.60e-07 (2.79) & 2.44e-08 (2.72) \\
 \hline
    $\ell$ & $1$ & $2$ & $3$ & $4$ \\
\hline
  \end{tabular}
\end{table}

In order to better appreciate this unexpected superconvergence result, we
report the results obtained after further refinements in
Table~\ref{tb:squarefine}. 

\begin{table}[t]
\caption{Convergence analysis for the L-shaped domain with non stabilized mass
matrix: $k=0$ and square mesh sequence}
\label{tb:squarefine}
\footnotesize
\centering
\begin{tabular}{c|ccccc}
\hline
Exact&\multicolumn{5}{c}{Relative error (rate)}\\[3pt]
\hline
1.47562 & 7.32e-05 & 1.84e-05 (2.00) & 4.60e-06 (2.00) & 1.15e-06 (2.00) & 2.88e-07 (2.00)\\
3.53403 & 1.15e-04 & 2.86e-05 (2.00) & 7.16e-06 (2.00) & 1.79e-06 (2.00) & 4.47e-07 (2.00)\\
9.86960 & 4.02e-04 & 1.00e-04 (2.00) & 2.51e-05 (2.00) & 6.27e-06 (2.00) & 1.57e-06 (2.00)\\
9.86960 & 4.02e-04 & 1.00e-04 (2.00) & 2.51e-05 (2.00) & 6.27e-06 (2.00) & 1.57e-06 (2.00)\\
11.38948 & 3.55e-04 & 8.88e-05 (2.00) & 2.22e-05 (2.00) & 5.55e-06 (2.00) & 1.39e-06 (2.00)\\
\hline
$\ell$ &4   & 5   & 6   & 7   & 8 \\
% Mesh & 64 & 128 & 256 & 512 & 1024\\
\hline
\end{tabular}
\end{table}

%===============================================================================
%===============================================================================

\subsubsection{Eigenvalues with stabilized mass matrix}

%===============================================================================
%===============================================================================

The last test we are going to show are related to triangular and square
mesh sequence. Tables~\ref{tb:s_L_tri} and~\ref{tb:s_L_squ} report the
corresponding results for $\sigma_E=0.1$ when $k$ goes from $0$ to $2$.
We have also numerically investigated cases which go beyond our theory. The
tests indicate that when using higher order methods stabilization is not
required even on general polygonal meshes.

\begin{table}[h]
 \caption{Triangular mesh on the L-shaped domain: computations with
stabilized mass matrix for $\sigma_E=0.1$}
\label{tb:s_L_tri}
\footnotesize
    \centering
    \begin{tabular}{c|c c c c c c c c c c}
    \hline
    Exact&\multicolumn{4}{c}{Relative error (rate)}\\[3pt]
\hline
\multicolumn{5}{c}{$k=0$}\\
\hline
   1.475622  & 1.53e-02  & 5.99e-03 (1.36) & 2.35e-03 (1.35) & 9.25e-04 (1.34) \\ 
   3.534031  & 1.70e-03  & 4.02e-04 (2.08) & 9.71e-05 (2.05) & 2.37e-05 (2.03) \\ 
   9.869604  & 1.55e-03  & 3.79e-04 (2.03) & 9.44e-05 (2.01) & 2.36e-05 (2.00) \\ 
   9.869604  & 7.56e-04  & 1.99e-04 (1.92) & 5.04e-05 (1.98) & 1.27e-05 (1.99) \\ 
   11.389479 & 5.77e-03  & 1.44e-03 (2.01) & 3.57e-04 (2.01) & 8.91e-05 (2.00) \\ 
 \hline 
\multicolumn{5}{c}{$k=1$}\\
\hline
  1.475622  & 1.39e-03  & 5.52e-04 (1.34) & 2.20e-04 (1.33) & 8.72e-05 (1.33) \\ 
  3.534031  & 6.24e-06  & 4.85e-07 (3.69) & 4.36e-08 (3.48) & 4.82e-09 (3.17) \\ 
  9.869604  & 1.04e-05  & 6.43e-07 (4.02) & 3.99e-08 (4.01) & 2.49e-09 (4.00) \\ 
  9.869604  & 2.21e-05  & 1.37e-06 (4.01) & 8.53e-08 (4.01) & 5.32e-09 (4.00) \\ 
  11.389479 & 5.07e-05  & 3.19e-06 (3.99) & 2.03e-07 (3.97) & 1.35e-08 (3.91)\\ 
\hline 
\multicolumn{5}{c}{$k=2$}\\
\hline
  1.475622  & 4.87e-04  & 1.93e-04 (1.33) & 7.68e-05 (1.33) & 3.05e-05 (1.33) \\ 
  3.534031  & 4.04e-07  & 6.24e-08 (2.69) & 9.78e-09 (2.67) & 1.55e-09 (2.66) \\ 
  9.869604  & 2.28e-08  & 3.64e-10 (5.97) & 5.65e-12 (6.01) & 5.81e-14 (6.60) \\ 
  9.869604  & 2.57e-08  & 4.08e-10 (5.98) & 6.40e-12 (5.99) & 3.52e-13 (4.18) \\ 
  11.389479 & 2.59e-07  & 2.87e-08 (3.18) & 4.34e-09 (2.72) & 6.89e-10 (2.66) \\
 \hline
    $\ell$ & $1$ & $2$ & $3$ & $4$ \\
\hline
  \end{tabular}
\end{table}

\begin{table}[h]
 \caption{Square mesh on the L-shaped domain: computations with
stabilized mass matrix for $\sigma_E=0.1$}
\label{tb:s_L_squ}
\footnotesize
    \centering
    \begin{tabular}{c|c c c c c c c c c c}
    \hline
    Exact&\multicolumn{4}{c}{Relative error (rate)}\\[3pt]
\hline
\multicolumn{5}{c}{$k=0$}\\
\hline
  1.475622  & 1.70e-03  & 1.16e-03 (0.56) & 5.80e-04 (1.00) & 2.60e-04 (1.15) \\ 
  3.534031  & 5.18e-03  & 1.29e-03 (2.01) & 3.22e-04 (2.00) & 8.03e-05 (2.00) \\ 
  9.869604  & 1.82e-02  & 4.51e-03 (2.01) & 1.13e-03 (2.00) & 2.81e-04 (2.00) \\ 
  9.869604  & 1.82e-02  & 4.51e-03 (2.01) & 1.13e-03 (2.00) & 2.81e-04 (2.00) \\ 
  11.389479 & 1.61e-02  & 3.99e-03 (2.01) & 9.95e-04 (2.00) & 2.49e-04 (2.00) \\ 
 \hline 
\multicolumn{5}{c}{$k=1$}\\
\hline
   1.475622  & 8.65e-04  & 3.45e-04 (1.32) & 1.37e-04 (1.33) & 5.45e-05 (1.33) \\ 
   3.534031  & 6.55e-06  & 4.90e-07 (3.74) & 4.20e-08 (3.54) & 4.38e-09 (3.26) \\ 
   9.869604  & 3.28e-05  & 2.06e-06 (3.99) & 1.29e-07 (4.00) & 8.06e-09 (4.00) \\ 
   9.869604  & 3.28e-05  & 2.06e-06 (3.99) & 1.29e-07 (4.00) & 8.06e-09 (4.00) \\ 
   11.389479 & 6.29e-05  & 4.00e-06 (3.98) & 2.56e-07 (3.97) & 1.68e-08 (3.93) \\ 
\hline 
\multicolumn{5}{c}{$k=2$}\\
\hline
   1.475622 & 1.94e-03  & 7.72e-04 (1.33) & 3.07e-04 (1.33) & 1.22e-04 (1.33) \\ 
   3.534031 & 1.37e-05  & 2.16e-06 (2.66) & 3.41e-07 (2.67) & 5.37e-08 (2.67) \\ 
   9.869604 & 7.23e-08  & 1.14e-09 (5.99) & 1.78e-11 (6.00) & 1.98e-13 (6.49) \\ 
   9.869604 & 7.23e-08  & 1.14e-09 (5.99) & 1.78e-11 (6.00) & 2.32e-13 (6.26) \\ 
  11.389479 & 8.56e-06  & 1.11e-06 (2.95) & 1.60e-07 (2.79) & 2.44e-08 (2.72) \\ 
 \hline
    $\ell$ & $1$ & $2$ & $3$ & $4$ \\
\hline
 \end{tabular}
\end{table}

\section{Conclusion} % (fold)
\label{sec:conclusion}
We have conducted an analysis of the virtual element method for the acoustic vibration problem and successfully demonstrated the convergence of the solution operator. This analysis was based on the equivalence between the original problem and a mixed system, which was then utilized in our analysis. The key element in proving convergence was the SDCP. Notably, we established that in certain cases, particularly with the use of lowest-order schemes, stabilization measures are not required. Our findings have been validated through several numerical tests, which align with our theoretical results.
% section conclusion (end)

\bibliographystyle{plain}
\bibliography{ref}

\begin{thebibliography}{10}

\bibitem{plate3}
Dibyendu Adak, David Mora, and Iv\'{a}n Vel\'{a}squez.
\newblock A {$C^0$}-nonconforming virtual element methods for the vibration and
  buckling problems of thin plates.
\newblock {\em Comput. Methods Appl. Mech. Engrg.}, 403:Paper No. 115763, 24,
  2023.

\bibitem{dunegridpaperII}
P.~Bastian, M.~Blatt, A.~Dedner, C.~Engwer, R.~Kl{\"o}fkorn, R.~Kornhuber,
  M.~Ohlberger, and O.~Sander.
\newblock A generic grid interface for parallel and adaptive scientific
  computing. part {II}: Implementation and tests in {DUNE}.
\newblock {\em Computing}, 82(2--3):121--138, 2008.

\bibitem{basic_principle}
L.~Beir\~{a}o~da Veiga, F.~Brezzi, A.~Cangiani, G.~Manzini, L.~D. Marini, and
  A.~Russo.
\newblock Basic principles of virtual element methods.
\newblock {\em Math. Models Methods Appl. Sci.}, 23(1):199--214, 2013.

\bibitem{Rodolfo}
Louren\c{c}o Beir\~{a}o~da Veiga, David Mora, Gonzalo Rivera, and Rodolfo
  Rodr\'{\i}guez.
\newblock A virtual element method for the acoustic vibration problem.
\newblock {\em Numer. Math.}, 136(3):725--763, 2017.

\bibitem{Berrone2021}
Stefano Berrone, Andrea Borio, and Francesca Marcon.
\newblock Lowest order stabilization free virtual element method for the {2D}
  {P}oisson equation.
\newblock arXiv:2103.16896 [math.NA], 2021.

\bibitem{Berrone2022}
Stefano Berrone, Andrea Borio, and Francesca Marcon.
\newblock Comparison of standard and stabilization free virtual elements on
  anisotropic elliptic problems.
\newblock {\em Appl. Math. Lett.}, 129:Paper No. 107971, 5, 2022.

\bibitem{Berrone2023}
Stefano Berrone, Andrea Borio, Francesca Marcon, and Gioana Teora.
\newblock A first-order stabilization-free virtual element method.
\newblock {\em Appl. Math. Lett.}, 142:Paper No. 108641, 6, 2023.

\bibitem{DBNote_DeRham}
D.~Boffi.
\newblock A note on the derham complex and a discrete compactness property.
\newblock {\em Applied Mathematics Letters}, 14(1):33--38, 2001.

\bibitem{bbg1}
D.~Boffi, F.~Brezzi, and L.~Gastaldi.
\newblock On the problem of spurious eigenvalues in the approximation of linear
  elliptic problems in mixed form.
\newblock {\em Math. Comp.}, 69(229):121--140, 2000.

\bibitem{bfgp}
D.~Boffi, P.~Fernandes, L.~Gastaldi, and I.~Perugia.
\newblock Computational models of electromagnetic resonators: analysis of edge
  element approximation.
\newblock {\em SIAM J. Numer. Anal.}, 36(4):1264--1290, 1999.

\bibitem{SDCP}
Daniele Boffi.
\newblock Approximation of eigenvalues in mixed form, discrete compactness
  property, and application to {$hp$} mixed finite elements.
\newblock {\em Comput. Methods Appl. Mech. Engrg.}, 196(37-40):3672--3681,
  2007.

\bibitem{Acta}
Daniele Boffi.
\newblock Finite element approximation of eigenvalue problems.
\newblock {\em Acta Numer.}, 19:1--120, 2010.

\bibitem{bbf}
Daniele Boffi, Franco Brezzi, and Michel Fortin.
\newblock {\em Mixed finite element methods and applications}, volume~44 of
  {\em Springer Series in Computational Mathematics}.
\newblock Springer, Heidelberg, 2013.

\bibitem{BBG2}
Daniele Boffi, Franco Brezzi, and Lucia Gastaldi.
\newblock On the convergence of eigenvalues for mixed formulations.
\newblock volume~25, pages 131--154 (1998). 1997.
\newblock Dedicated to Ennio De Giorgi.

\bibitem{Calcolo}
Daniele Boffi, Francesca Gardini, and Lucia Gastaldi.
\newblock Approximation of {PDE} eigenvalue problems involving parameter
  dependent matrices.
\newblock {\em Calcolo}, 57(4):Paper No. 41, 21, 2020.

\bibitem{Sema}
Daniele Boffi, Francesca Gardini, and Lucia Gastaldi.
\newblock Virtual element approximation of eigenvalue problems.
\newblock In {\em The virtual element method and its applications}, volume~31
  of {\em SEMA SIMAI Springer Ser.}, pages 275--320. Springer, Cham, [2022]
  \copyright 2022.

\bibitem{basic_mixed}
Franco Brezzi, Richard~S Falk, and L~Donatella Marini.
\newblock Basic principles of mixed virtual element methods.
\newblock {\em ESAIM: Mathematical Modelling and Numerical Analysis},
  48(4):1227--1240, 2014.

\bibitem{Chen}
Alvin Chen and N.~Sukumar.
\newblock Stabilization-free serendipity virtual element method for plane
  elasticity.
\newblock {\em Comput. Methods Appl. Mech. Engrg.}, 404:Paper No. 115784, 19,
  2023.

\bibitem{l_shape_benchmark}
Monique Dauge.
\newblock Benchmark computations for {M}axwell equations for the approximation
  of highly singular solutions.
\newblock
  \url{https://perso.univ-rennes1.fr/monique.dauge/benchmax.html#3.2DomA_EP}.
\newblock Accessed: 2023-08-21.

\bibitem{dunevem}
Andreas Dedner and Alice Hodson.
\newblock A framework for implementing general virtual element spaces.
\newblock {\em to appear in SIAM J. Sci. Comput.}, 2024.

\bibitem{dune}
Andreas Dedner, Robert Kloefkorn, and Martin Nolte.
\newblock Python bindings for the {DUNE-FEM} module.
\newblock {\em Zenodo. doi}, 10, 2020.

\bibitem{Gardininonc}
Francesca Gardini, Gianmarco Manzini, and Giuseppe Vacca.
\newblock The nonconforming virtual element method for eigenvalue problems.
\newblock {\em ESAIM Math. Model. Numer. Anal.}, 53(3):749--774, 2019.

\bibitem{GardiniVacca}
Francesca Gardini and Giuseppe Vacca.
\newblock Virtual element method for second-order elliptic eigenvalue problems.
\newblock {\em IMA J. Numer. Anal.}, 38(4):2026--2054, 2018.

\bibitem{GR}
Vivette Girault and Pierre-Arnaud Raviart.
\newblock {\em Finite element methods for {N}avier-{S}tokes equations},
  volume~5 of {\em Springer Series in Computational Mathematics}.
\newblock Springer-Verlag, Berlin, 1986.
\newblock Theory and algorithms.

\bibitem{slepc}
Vicente Hernandez, Jose~E. Roman, and Vicente Vidal.
\newblock S{LEP}c: a scalable and flexible toolkit for the solution of
  eigenvalue problems.
\newblock {\em ACM Trans. Math. Software}, 31(3):351--362, 2005.

\bibitem{kik}
Fumio Kikuchi.
\newblock Mixed and penalty formulations for finite element analysis of an
  eigenvalue problem in electromagnetism.
\newblock {\em Comput. Methods Appl. Mech. Engrg.}, 64(1-3):509--521, 1987.

\bibitem{kik_weak}
Fumio Kikuchi.
\newblock Weak formulations for finite element analysis of an electromagnetic
  eigenvalue problem.
\newblock {\em Sci. Papers College Arts Sci., Univ. Tokyo}, 38:43--67, 1988.

\bibitem{Kikuchi}
Fumio Kikuchi.
\newblock On a discrete compactness property for the {N}\'{e}d\'{e}lec finite
  elements.
\newblock {\em J. Fac. Sci. Univ. Tokyo Sect. IA Math.}, 36(3):479--490, 1989.

\bibitem{Lamperti}
Andrea Lamperti, Massimiliano Cremonesi, Umberto Perego, Alessandro Russo, and
  Carlo Lovadina.
\newblock A {H}u-{W}ashizu variational approach to self-stabilized virtual
  elements: 2{D} linear elastostatics.
\newblock {\em Comput. Mech.}, 71(5):935--955, 2023.

\bibitem{Steklov3}
Felipe Lepe, David Mora, Gonzalo Rivera, and Iv\'{a}n Vel\'{a}squez.
\newblock A virtual element method for the {S}teklov eigenvalue problem
  allowing small edges.
\newblock {\em J. Sci. Comput.}, 88(2):Paper No. 44, 21, 2021.

\bibitem{plate2}
Jian Meng and Liquan Mei.
\newblock A mixed virtual element method for the vibration problem of clamped
  {K}irchhoff plate.
\newblock {\em Adv. Comput. Math.}, 46(5):Paper No. 68, 18, 2020.

\bibitem{transmission4}
Jian Meng and Liquan Mei.
\newblock Virtual element method for the {H}elmholtz transmission eigenvalue
  problem of anisotropic media.
\newblock {\em Math. Models Methods Appl. Sci.}, 32(8):1493--1529, 2022.

\bibitem{transmission3}
Jian Meng, Gang Wang, and Liquan Mei.
\newblock A lowest-order virtual element method for the {H}elmholtz
  transmission eigenvalue problem.
\newblock {\em Calcolo}, 58(1):Paper No. 2, 22, 2021.

\bibitem{Meng}
Jian Meng, Xue Wang, Linlin Bu, and Liquan Mei.
\newblock A lowest-order free-stabilization virtual element method for the
  {L}aplacian eigenvalue problem.
\newblock {\em J. Comput. Appl. Math.}, 410:Paper No. 114013, 11, 2022.

\bibitem{VEMmixed}
Jian Meng, Yongchao Zhang, and Liquan Mei.
\newblock A virtual element method for the {L}aplacian eigenvalue problem in
  mixed form.
\newblock {\em Appl. Numer. Math.}, 156:1--13, 2020.

\bibitem{MonkDemko}
P.~Monk and L.~Demkowicz.
\newblock Discrete compactness and the approximation of {M}axwell's equations
  in {${\Bbb R}^3$}.
\newblock {\em Math. Comp.}, 70(234):507--523, 2001.

\bibitem{elasticity}
David Mora and Gonzalo Rivera.
\newblock {\it {A} priori} and {\it a posteriori} error estimates for a virtual
  element spectral analysis for the elasticity equations.
\newblock {\em IMA J. Numer. Anal.}, 40(1):322--357, 2020.

\bibitem{Steklov1}
David Mora, Gonzalo Rivera, and Rodolfo Rodr\'{\i}guez.
\newblock A virtual element method for the {S}teklov eigenvalue problem.
\newblock {\em Math. Models Methods Appl. Sci.}, 25(8):1421--1445, 2015.

\bibitem{Steklov2}
David Mora, Gonzalo Rivera, and Rodolfo Rodr\'{\i}guez.
\newblock A posteriori error estimates for a virtual element method for the
  {S}teklov eigenvalue problem.
\newblock {\em Comput. Math. Appl.}, 74(9):2172--2190, 2017.

\bibitem{plate1}
David Mora, Gonzalo Rivera, and Iv\'{a}n Vel\'{a}squez.
\newblock A virtual element method for the vibration problem of {K}irchhoff
  plates.
\newblock {\em ESAIM Math. Model. Numer. Anal.}, 52(4):1437--1456, 2018.

\bibitem{transmission1}
David Mora and Iv\'{a}n Vel\'{a}squez.
\newblock A virtual element method for the transmission eigenvalue problem.
\newblock {\em Math. Models Methods Appl. Sci.}, 28(14):2803--2831, 2018.

\bibitem{plate4}
David Mora and Iv\'{a}n Vel\'{a}squez.
\newblock Virtual element for the buckling problem of {K}irchhoff-{L}ove
  plates.
\newblock {\em Comput. Methods Appl. Mech. Engrg.}, 360:112687, 22, 2020.

\bibitem{transmission2}
David Mora and Iv\'{a}n Vel\'{a}squez.
\newblock Virtual elements for the transmission eigenvalue problem on polytopal
  meshes.
\newblock {\em SIAM J. Sci. Comput.}, 43(4):A2425--A2447, 2021.

\bibitem{Gardinihp}
O.~\v{C}ert\'{\i}k, F.~Gardini, G.~Manzini, L.~Mascotto, and G.~Vacca.
\newblock The {$p$}- and {$hp$}-versions of the virtual element method for
  elliptic eigenvalue problems.
\newblock {\em Comput. Math. Appl.}, 79(7):2035--2056, 2020.

\bibitem{Steklov4}
Gang Wang, Jian Meng, Ying Wang, and Liquan Mei.
\newblock {\it {A} priori} and {\it a posteriori error} estimates for a virtual
  element method for the non-self-adjoint {S}teklov eigenvalue problem.
\newblock {\em IMA J. Numer. Anal.}, 42(4):3675--3710, 2022.

\end{thebibliography}

% section function_spaces (end)
\end{document}